\documentclass{imanum2}

\usepackage{subcaption}
\usepackage{pgfplots}

\newcommand{\SlopeTriangle}[6]
{

    \pgfplotsextra
    {
        \pgfkeysgetvalue{/pgfplots/xmin}{\xmin}
        \pgfkeysgetvalue{/pgfplots/xmax}{\xmax}
        \pgfkeysgetvalue{/pgfplots/ymin}{\ymin}
        \pgfkeysgetvalue{/pgfplots/ymax}{\ymax}

        \pgfmathsetmacro{\xArel}{#1}
        \pgfmathsetmacro{\yArel}{#3}
        \pgfmathsetmacro{\xBrel}{#1-#2}
        \pgfmathsetmacro{\yBrel}{\yArel}
        \pgfmathsetmacro{\xCrel}{\xArel}

        \pgfmathsetmacro{\lnxB}{\xmin*(1-(#1-#2))+\xmax*(#1-#2)} 
        \pgfmathsetmacro{\lnxA}{\xmin*(1-#1)+\xmax*#1} 
        \pgfmathsetmacro{\lnyA}{\ymin*(1-#3)+\ymax*#3} 
        \pgfmathsetmacro{\lnyC}{\lnyA+#4*(\lnxA-\lnxB)}
        \pgfmathsetmacro{\yCrel}{\lnyC-\ymin)/(\ymax-\ymin)} 

        \coordinate (A) at (rel axis cs:\xArel,\yArel);
        \coordinate (B) at (rel axis cs:\xBrel,\yBrel);
        \coordinate (C) at (rel axis cs:\xCrel,\yCrel);

        \draw[#6]   (A)-- node[anchor=north] {#5}
                    (B)--
                    (C)--
                    cycle;
    }
}

\numberwithin{figure}{section}

\jno{drnxxx}

\usepackage{hyperref}
\definecolor{myblue}{rgb}{0,0,0.5}
\definecolor{myred}{rgb}{0.6,0,0}
\hypersetup{colorlinks=true,linkcolor=myblue,citecolor=myblue,filecolor=myblue,urlcolor=myblue}

\usepackage{mathrsfs}

\renewcommand{\Re}{\operatorname{Re}}

\newcommand{\Cquad}{\gamma_{\rm q}}
\newcommand{\Cstab}{\gamma_{\rm s}^{*}}

\newcommand{\Caprx}{\gamma_{\rm a}^*}
\newcommand{\Caprxt}{\widetilde{\gamma}_{\rm a}^*}
\newcommand{\Caprxop}{\gamma_{\rm a}^+}
\newcommand{\Cdrhs}{\widetilde{\gamma}_{\rm q}}

\newcommand{\MI}{\textup{I}_d}

\newcommand{\Ccontt}{\Ccont^+}
\newcommand{\newb}{b^+}
\newcommand{\tsol}{\mathcal{S}^+}
\newcommand{\tpi}{\pi^+}

\newcommand{\Ccont}{M}
\newcommand{\Ccoer}{C_{\rm coer}}

\newcommand{\asol}{\mathcal S^*}

\newcommand{\eq}{:=}

\newcommand{\Rea}{\mathbb{R}}

\newcommand{\beq}{\begin{equation}}
\newcommand{\eeq}{\end{equation}}
\newcommand{\beqs}{\begin{equation*}}
\newcommand{\eeqs}{\end{equation*}}

\newcommand{\bpf}{\begin{proof}}
\newcommand{\epf}{\end{proof}}

\newcommand{\tfa}{\text{ for all }}
\newcommand{\tfor}{\text{ for }}

\newcommand{\tif}{\text{ if }}
\newcommand{\tin}{\text{ in }}
\newcommand{\ton}{\text{ on }}

\newcommand{\tand}{\text{ and }}

\newcommand{\PMLsol}{u}

\newcommand{\tb}{\widetilde b}

\newcommand{\tu}{\widetilde u}
\newcommand{\tv}{\widetilde v}
\newcommand{\tw}{\widetilde w}

\newcommand{\tV}{\widetilde V}
\newcommand{\tCI}{\widetilde \CI}

\newcommand{\GD}{\Gamma_{\rm D}}
\newcommand{\GN}{\Gamma_{\rm N}}

\newcommand{\MAP}{\mathcal F}
\newcommand{\IMAP}{\mathcal G}
\newcommand{\CT}{\mathcal T}
\newcommand{\CP}{\mathcal P}
\newcommand{\CI}{\mathcal I}
\newcommand{\CJ}{\mathcal J}
\newcommand{\tK}{\widetilde K}
\newcommand{\hK}{\widehat   K}

\newcommand{\Omegah}{\Omega^h}
\newcommand{\GDh}{\Gamma_{\rm D}^h}
\newcommand{\GNh}{\Gamma_{\rm N}^h}

\newcommand{\Part}{{\mathfrak{Q}}}
\newcommand{\Parth}{\Part^h}

\newcommand{\hv}{\widehat v}
\newcommand{\hx}{\widehat x}

\newcommand*{\N}[1]{\left\|#1\right\|}

\newcommand{\operator}{\mathcal{K}}
\newcommand{\cH}{\mathcal{H}}

\newcommand{\Omegat}{\Omega_{\rm tot}}

\newcommand{\datai}{f}
\newcommand{\data}{g}

\newcommand{\InformalTheorem}{Main Result}
\newtheoremstyle{informaltheorem}{6pt}{6pt}{\rm}{}{\sffamily}{ }{ }{}
\theoremstyle{informaltheorem}
\newtheorem{informaltheorem}{\sc Main Result}[section]

\begin{document}

\title{The geometric error is less than the pollution error when solving the high-frequency Helmholtz equation with high-order FEM on curved domains}

\author{T. Chaumont-Frelet\thanks{Inria Univ.~Lille and Laboratoire Paul Painlev\'e, 59655 Villeneuve-d'Ascq, France, {\tt theophile.chaumont@inria.fr}},
\,\,E. A. Spence\thanks{Department of Mathematical Sciences, University of Bath, Bath, BA2 7AY, UK, \tt E.A.Spence@bath.ac.uk}
}

\shortauthorlist{T. Chaumont-Frelet and E.A. Spence}

\maketitle

\begin{abstract}
{
We consider the $h$-version of the finite-element method, 
where accuracy is increased by decreasing the meshwidth $h$ while keeping the polynomial degree $p$ constant,
applied to the Helmholtz equation.
Although the question ``how quickly must $h$ decrease as the wavenumber $k$ increases to maintain accuracy?''~has been studied intensively since the 1990s, 
none of the existing rigorous wavenumber-explicit analyses take into account the approximation of the geometry. 
In this paper we prove that for nontrapping problems solved using straight elements the geometric error is order $kh$, which is then less than the pollution error $k(kh)^{2p}$ when $k$ is large; this fact is then illustrated in numerical experiments. More generally, we prove that, even for problems with strong trapping, using degree four (in 2-d) or degree five (in 3-d) polynomials and isoparametric elements ensures that the geometric error is smaller than the pollution error for most large wavenumbers.
}
{
Helmholtz equation, time-harmonic scattering, high-frequency problems,
curved finite elements, geometric error, duality techniques, error estimates
}
\end{abstract}

\section{Introduction:~informal statement of the main result}\label{sec:intro}

The results of this paper apply to general Helmholtz problems (interior problems, exterior problems, wave-guide problems), but for concreteness here 
we consider the Helmholtz equation 
\beq\label{eq:Helmholtz}
-k^2 \mu u  -\nabla \cdot (A\nabla u ) =\datai,
\eeq
with $k\gg 1$, 
posed in a bounded domain $\Omega$ corresponding to the exterior of an impenetrable obstacle where the Sommerfeld radiation condition has been approximated by a radial perfectly-matched layer (PML). For simplicity, we assume that $\Omega$ has characteristic length scale $\sim 1$. 
The solution $u$ satisfies either a zero Dirichlet or zero Neumann boundary condition on the part of $\partial \Omega$ corresponding to the impenetrable obstacle, and a zero Dirichlet boundary condition on the part of $\partial\Omega$ corresponding to the PML truncation boundary. 
The piecewise-smooth coefficients $\mu$ and $A$ describe the PML near the truncation boundary, and variation of $\mu$ and $A$ in the rest of the domain corresponds to penetrable obstacles. 

We assume that the solution operator to \eqref{eq:Helmholtz} is bounded by $k^{\alpha-1}$ for some $\alpha>0$, i.e., given $k_0>0$ there exists $C>0$ such that given $\datai\in L^2(\Omega)$, the solution  $u\in H^1(\Omega)$  to \eqref{eq:Helmholtz} satisfies
\beq\label{eq:alpha}
\N{u}_{H^1_k(\Omega)}\leq C k^{\alpha-1} \N{\datai}_{L^2(\Omega)} \text{ for all } k\geq k_0,
\eeq
where 
\beq\label{eq_norm}
\N{u}_{H^1_k(\Omega)}^2:= \N{\nabla u}_{L^2(\Omega)}^2 + k^2 \N{u}^2_{L^2(\Omega)}.
\eeq
Recall that the bound 
\eqref{eq:alpha} holds with $\alpha=1$ 
when the problem is \emph{nontrapping} (i.e., all geometric-optic rays starting in a neighbourhood of the scatterer escape to infinity in a uniform time); indeed 
in this case the solution of the exterior Helmholtz problem satisfies \eqref{eq:alpha} with $\alpha=1$ by \cite{Mo:75, Va:75}, 
and the solution of the associated radial PML problem inherits this bound by \cite[Lemma 3.3]{GLS2}.

\begin{informaltheorem}[Informal statement of the main result]\label{thm:informal1}
Suppose that the Helmholtz problem above is solved using the $h$-version of the finite-element method with shape-regular meshes with meshwidth $h$, fixed (but arbitrary) polynomial degree $p$, and fixed order of the geometric approximation $q$ (so that $q=1$ for straight elements and $q=p$ for isoparametric elements). 

Suppose further that 
the domain $\Omega$ and coefficients $A$ and $\mu$
satisfy the natural regularity assumptions for using degree $p$ polynomials, and that the data $\datai$ in \eqref{eq:Helmholtz} 
comes from an incident plane wave or point source (or, more generally, a particular Helmholtz solution that is smooth in a neighbourhood of the scatterer).
If 
\beq\label{eq:condition1}
k^{\alpha}(kh)^{2p}
+ k^{\alpha} h^q \,\,\text{ is sufficiently small, }
\eeq
then the Galerkin solution $u_h$ exists, is unique, and satisfies
\beq\label{eq:informal1}
\frac{
\N{u-u_h}_{H^1_k(\Omega)}
}{
\N{u}_{H^1_k(\Omega)}
}
\leq C \bigg[ 
 \Big( 1 +  k^{\alpha}(kh)^p\Big) (kh)^p + k^{\alpha} h^q
\bigg].
\eeq
\end{informaltheorem}

We make eight remarks about the interpretation and context of \InformalTheorem \ref{thm:informal1}.
\begin{enumerate}
\item 
Near $\partial \Omega$, the error $\|u-u_h\|_{H^1_k(\Omega)}$ in \eqref{eq:informal1} is measured on the subset of $\Omega$ where both $u$ and $u_h$ are well-defined. When $A$ and $\mu$ are discontinuous, defining precisely where the error $\|u-u_h\|_{H^1_k(\Omega)}$ is measured is more involved; see Theorem \ref{thm_main_applied} below.

\item In the limit $h\to 0$ with $k$ fixed, the bound \eqref{eq:informal1} shows that the Galerkin error is $O(h^{\min\{p,q\}})$; this is expected, at least in the isoparametric case when $p=q$, by \cite{CiRa:72, lenoir_1986a}.

\item When $q=\infty$ (i.e., there is no geometric error), the condition \eqref{eq:condition1} becomes 
\beqs
k^\alpha(kh)^{2p} \,\,\text{ is sufficiently small},
\eeqs
and then \eqref{eq:informal1} implies that choosing $k^\alpha(kh)^{2p}$ to be a sufficiently-small constant maintains accuracy of the Galerkin solution as $k\to \infty$. 
This result was
\begin{itemize}
\item proved for constant-coefficient problems in 1-d by \cite{IhBa:97} (see \cite[Page 350, penultimate displayed equation]{IhBa:97}, \cite[Equation 4.7.41]{Ih:98}), 
\item 
obtained in the context of dispersion analysis by \cite{Ai:04}, 
\item proved for 2- and 3-d constant-coefficient Helmholtz problems
with the radiation condition approximated by an impedance boundary condition
by \cite{DuWu:15} (following \cite{FeWu:09, FeWu:11, Wu:14, ZhWu:13})
and for variable-coefficient problems in \cite[Theorem 2.39]{Pe:20}, and then 
\item proved for general 2- and 3-d Helmholtz problems with truncation either by a PML, or the exact Dirichlet-to-Neumann map, or an impedance boundary condition
 by \cite{GaSp:23}.
\end{itemize}
We note that bounds under the condition ``$k^{\alpha} (kh)^p$ sufficiently small" were proved for several specific Helmholtz problems by  \cite{MeSa:10, MeSa:11},
and then for general Helmholtz problems by  \cite{ChNi:20}.

\item When the problem is nontrapping, $\alpha=1$, and 
$k h^q  \ll 1$ when $k(kh)^{2p}$ is constant, regardless of $q$. Therefore, 
\InformalTheorem \ref{thm:informal1} implies that, regardless of $q$,
choosing $k^\alpha(kh)^{2p}$ to be a sufficiently-small constant maintains accuracy of the Galerkin solution as $k\to \infty$. That is, even with straight elements ($q=1$),
in the limit $k\to \infty$ with $h$ chosen to control the pollution error, 
the geometric error is smaller than the pollution error; this feature is illustrated in the numerical experiments in \S\ref{sec:num}.

\item When the problem is trapping, $\alpha >1$, and it is not immediate that $k^\alpha h^q  \ll 1$
when $k^\alpha(kh)^{2p}$ is constant. For Helmholtz problems with strong trapping, the solution
operator grows exponentially through an increasing sequence of wavenumbers, with this sequence
becoming increasingly dense as $k$ increases; see \cite{PoVo:99}, \cite{CaPo:02} and
\cite{BeChGrLaLi:11}. However, if a set of wavenumbers of arbitrarily-small measure is
excluded then \eqref{eq:alpha} holds for any $\alpha > 5d/2 + 2$ by \cite{LSW1} and
\cite[Lemma 3.3]{GLS2} (in fact, \cite{LSW1} proves this result for the original exterior Helmholtz
problem, and the PML problem inherits this bound by \cite[Lemma 3.3]{GLS2}). With isoparametric
elements $q=p$, and then $h^q k \ll 1$ when $k^\alpha(kh)^{2p}$ is constant
if $p\geq 4$ in 2-d and $p\geq 5$ in 3-d. That is, for moderate polynomial degree,
isoparametric elements, and most (large) wavenumbers, the geometric error is smaller than
the pollution error, even for problems with strong trapping. 

\item To prove \InformalTheorem \ref{thm:informal1}, we extend the duality argument recently introduced
for general Helmholtz problems in \cite{GaSp:23} to incorporate a Strang-lemma-type argument to
allow variational crimes; see Theorem \ref{thm_abs1} below. We then use the results of \cite{lenoir_1986a} to apply this abstract
result to the case of geometric error induced by polynomial element maps.

\item
\InformalTheorem \ref{thm:informal1} is conceptually similar to the results of \cite{ChNi:18, ChNi:23}. Indeed, \cite{ChNi:18, ChNi:23} 
show that, when solving 2-d or 3-d Helmholtz problems on domains with corners/conical points using a uniform mesh, the error from the corner singularity is smaller than the pollution error (for $k$ sufficiently large). 
In other words, although \emph{both} the geometric error \emph{and} the error incurred by corner singularities are important in the limit $h\to 0$ with $k$ fixed, 
in the limit $k\to \infty$ with $h$ chosen to control the pollution error, 
both of these errors are smaller than the pollution error.

\item
Finally, we note that 
the geometric error in  \InformalTheorem \ref{thm:informal1} can be understand as 
a (suitably-scaled) norm of solution operator ($k^\alpha$) multiplied by the magnitude of the perturbation ($h^q$).
The arguments in this paper therefore do not take into account any of the propagative properties of solutions of the Helmholtz equation. Using these properties, 
one might hope to show that any variation in the geometric that is ``subwavelength'', i.e.., $\ll k^{-1}$, is not ``seen'' by either the solution or the FEM solution, in which case the geometric error would always be smaller than the pollution error (regardless of the norm of the solution operator or the polynomial degree). 
Investigating this question is work in progress, but far beyond the scope of the present paper.

\end{enumerate}

\paragraph{Outline of the paper.}

\S\ref{sec:num} gives numerical experiments illustrating \InformalTheorem \ref{thm:informal1} for two 2-d nontrapping problems ($\alpha=1$) and straight elements ($q=1$).
\S\ref{sec:abstract} contains an abstract analogue of \InformalTheorem \ref{thm:informal1} proved under the assumptions that the sesquilinear form is continuous and satisifies a G\aa rding-type inequality. \S\ref{sec:all_here} shows that a wide variety of Helmholtz problems fix into the abstract setting of \S\ref{sec:abstract} and proves \InformalTheorem \ref{thm:informal1}.

\section{Numerical experiments with straight elements in 2-d}\label{sec:num}

In this section we demonstrate numerically for two 2-d scattering problems the consequence of \InformalTheorem \ref{thm:informal1} discussed in Point 4 above; i.e., that,
for nontrapping problems solved using straight elements, 
in the limit $k\to \infty$ with $h$ chosen to control the pollution error, 
the geometric error is smaller than the pollution error.

\subsection{The two scattering problems considered.}

\begin{definition}[Sound-soft scattering by a 2-d ball]
Given $u_{\rm inc}(x) := e^{i kx_1}$, 
let $u_{\rm tot}$ satisfy 
\beqs
-k^2 u_{\rm tot}-\Delta u_{\rm tot} = 0 \quad \tin \Rea^d \setminus B_1(0), \quad u_{\rm tot} = 0 \quad\ton \partial B_1(0),
\eeqs
where $u_{\rm sca}:= u_{\rm tot} - u_{\rm inc}$ satisfies the Sommerfeld radiation condition
\beq\label{eq:src}
\left(\frac{\partial}{\partial r} - ik\right)u_{\rm sca} = o(r^{-(d-1)/2}) \quad\text{ as } r:=|x|\to \infty, \text{ uniformly in } x/r.
\eeq
\end{definition}

\begin{definition}[Scattering by a penetrable 2-d ball]
Given $u_{\rm inc}(x) := e^{i kx_1}$ and 
\beq\label{eq_trans_coeff}
A:= 
\begin{cases} 
2 & \tin B_1(0),\\
1 & \tin \Rea^d \setminus B_1(0),
\end{cases}
\quad
\tand
\quad
\mu:= 
\begin{cases} 
1/2 & \tin B_1(0),\\
1 & \tin \Rea^d \setminus B_1(0),
\end{cases}
\eeq
let $u_{\rm tot}$ satisfy 
\beq\label{eq_transmission_PDE}
-k^2\mu u_{\rm tot}-\nabla \cdot (A\nabla u_{\rm tot}) = 0 \quad \tin \Rea^d,
\eeq
and 
\beq\label{eq_transmission}
2 \partial_r u^+_{\rm tot} = \partial_r u^-_{\rm tot} \quad\tand \quad u^+_{\rm tot} = u^-_{\rm tot} \quad\ton \partial B_1(0),
\eeq
where the subscript $+$ denotes the limit taken with $r>1$ and the subscript $-$ denotes the limit taken with $r<1$, 
and where $u_{\rm sca}:= u_{\rm tot} - u_{\rm inc}$ satisfies the Sommerfeld radiation condition \eqref{eq:src}.
\end{definition}
(Note that the transmission conditions in \eqref{eq_transmission} follow from the variational formulation of the PDE \eqref{eq_transmission_PDE}.)

The choice of the coefficients $A$ and $\mu$ in \eqref{eq_trans_coeff} implies that this problem is nontrapping, in the sense that outgoing solutions to \eqref{eq:Helmholtz}, with  $A$ and $\mu$ given by \eqref{eq_trans_coeff}, satisfy the bound \eqref{eq:alpha} with $\alpha=1$; see \cite{MoSp:19}.

The reason for choosing these two scattering problems is that, in both cases,  $u_{\rm tot}$ can be expressed explicitly as a Fourier series in the angular variable, with the Fourier coefficients expressed in terms of Hankel and Bessel functions. 

Although it is possible to directly compute $u_{\rm tot}$ with the FEM,
it is more convenient to compute 
\beqs
u:= \chi u_{\rm inc} + u_{\rm scat}
\eeqs
for $\chi\in C^\infty_{\rm comp}(\Rea^d)$ and $\chi \equiv1$ in a neighbourhood of $B_1(0)$;
this is because $u$ then satisfies (i) the same boundary/transmission conditions as $u_{\rm tot}$
on $\partial B_1(0)$ and (ii) the Sommerfeld radiation condition (satisfied by $u_{\rm sca}$)
at infinity. Once $u$ is computed, $u_{\rm tot}$ and $u_{\rm sca}$ can easily be 
recovered since $\chi$ and $u_{\rm inc}$ are known.

We then use a PML to approximate the radiation condition satisfied by $u$. Following \cite{collino_monk_1998a}, we let
\beq\label{eq_mu_PML}
\mu_{\rm PML}(x) := \mathfrak B(r_{\rm P})\beta(r_{\rm P})
\eeq
and
\begin{equation}\label{eq_A_PML}
A_{\rm PML}(x)
=
\frac{\mathfrak B(|x|)}{\beta(|x|)}
\left (
\begin{array}{cc}
\cos^2 \theta & \cos\theta\sin\theta)
\\
\cos\theta\sin\theta & \sin^2\theta
\end{array}
\right )
+
\frac{\beta(|x|)}{\mathfrak B(|x|)}
\left (
\begin{array}{cc}
\sin^2 \theta & -\cos\theta\sin\theta
\\
-\cos\theta\sin\theta & \cos^2\theta
\end{array}
\right ),
\end{equation}
where 
$\theta$ is such that $x = |x|(\cos\theta,\sin\theta)$, and
\begin{equation*}
\beta(r)
:=
\begin{cases}
1+ i \alpha (r-4)^\ell, & r\geq 4,\\
1, & 0\leq r< 4,
\end{cases}
\quad
\tand
\quad
\mathfrak B(r)
:=
\begin{cases}
1+ i\frac{\alpha}{(\ell+1) r} (r-4)^{\ell+1}, & r \geq 4,\\
1, & 0\leq r<4,
\end{cases}
\end{equation*}
with $\alpha = 10$ and $\ell=2$ (note that our $\beta$ and $\mathfrak{B}$ correspond to $\gamma$ and $\overline{\gamma}$ in \cite{collino_monk_1998a}).

We therefore approximate with the FEM the solutions of the following two problems.

\begin{definition}
[PML approximation to sound-soft scattering by a 2-d ball]
\label{def_PML_approx_ss}
Let $\Omega:= B_5(0)\setminus \overline{B_1(0)}.$
Let 
\beqs
A:= 
\begin{cases} 
1 & \tin B_4(0)\setminus \overline{B_1(0)},\\
A_{\rm PML} & \tin B_5(0) \setminus B_4(0),
\end{cases}
\quad
\tand
\quad
\mu:= 
\begin{cases} 
1& \tin B_4(0)\setminus \overline{B_1(0)},\\
\mu_{\rm PML} & \tin B_5(0) \setminus B_4(0),
\end{cases}
\eeqs
where $A_{\rm PML}$ and $\mu_{\rm PML}$ are given by \eqref{eq_A_PML} and \eqref{eq_mu_PML} respectively.
Given $\chi \in C^\infty_{\rm comp}(\Rea^d)$ satisfying 
\beq\label{eq_chi}
\chi \equiv 1 \ton B_1(0)\quad\tand\quad \chi \equiv 0\ton B_5(0)\setminus B_4(0),
\eeq
let $u\in H^1_0(\Omega)$ satisfy 
\begin{equation}\label{eq_new_PDE}
-k^2 \mu u - \nabla \cdot(A\nabla u) = \datai
\end{equation}
with
\begin{equation}\label{eq_new_rhs}
\datai
:=
-(\Delta \chi) u_{\rm inc}-2\nabla \chi \cdot \nabla u_{\rm inc}
=
(-\Delta \chi-2ik\partial_1\chi) e^{ikx_1}.
\end{equation}
\end{definition}

\begin{definition}
[PML approximation to scattering by a penetrable 2-d ball]
\label{def_PML_approx_pen}
Let $\Omega:= B_5(0).$
Let 
\beqs
A:= 
\begin{cases} 
2 & \tin B_1(0),\\
1 & \tin B_4(0)\setminus B_1(0),\\
A_{\rm PML} & \tin B_5(0) \setminus B_4(0)
\end{cases}
\quad
\tand
\quad
\mu:= 
\begin{cases} 
1/2 & \tin B_1(0),\\
1 & \tin B_4(0)\setminus B_1(0),\\
\mu_{\rm PML} & \tin B_5(0) \setminus B_4(0)
\end{cases}
\eeqs
where $A_{\rm PML}$ and $\mu_{\rm PML}$ are given by \eqref{eq_A_PML} and \eqref{eq_mu_PML} respectively.
Given $\chi \in C^\infty_{\rm comp}(\Rea^d)$ satisfying \eqref{eq_chi}, 
let $u\in H^1_0(\Omega)$ satisfy \eqref{eq_new_PDE}, with $\datai$ given by \eqref{eq_new_rhs}, and the transmission conditions 
\beqs
2 \partial_r u^+ = \partial_r u^-\quad\tand \quad u^+ = u^-\quad\ton \partial B_1(0).
\eeqs
\end{definition}

For simplicity, in the numerical experiments below, we actually use the cutoff
\begin{equation*}
\chi(x) := \frac{1}{2} \left (1-\operatorname{erf}\left (\frac{|x|-3}{\sigma}\right )\right ),
\end{equation*}
with $\sigma = 0.2$ and
\begin{equation*}
\operatorname{erf}(t)
:=
\frac{2}{\sqrt{\pi}}
\int_0^t e^{-s^2} ds.
\end{equation*}
This $\chi$ does not satisfy \eqref{eq_chi};
however,
$|1-\chi(x)| < 10^{-12}$ if $|x| \leq 2$ and $|\chi(x)| < 10^{-12}$ if $|x| \geq 4$,
so that the error incurred by this inconsistency is so small that it does not
affect our examples.

\subsection{The finite-element solution and error measurement}

For simplicity, we only consider the error in the region where $\chi \equiv 1$, so that  $u \equiv u_{\rm tot}$, i.e., 
\beqs
\Omega_{\rm tot}:= 
\begin{cases}
B_2(0) \setminus B_1(0) & \text{ for the sound-soft ball},\\
B_2(0) & \text{ for the penetrable ball.}
\end{cases}
\eeqs
Crucially, $\Omega_{\rm tot}$ includes the interface $\partial B_1(0)$
where the geometric approximation takes place.

As a proxy for $u_{\rm tot}$ we use the Fourier-series solution; this incurs an error due to the PML approximating the radiation condition, but this error is very small, especially for large wavenumbers; see \cite{GLS2}. 

We consider finite element discretizations of degree $p=2,3$ or $4$ and straight
meshes. 
We use {\tt gmsh} (see \cite{geuzaine_remacle_2009a}) to generate the meshes.
{\tt gmsh} produces nodally conforming meshes (meaning that interfaces are
not crossed by any edge) of specified maximal size $h_{\rm target}$. Once such
a mesh has been produced, we compute the actual mesh size $h$ by measuring the
length of all mesh edges. This value of $h$ is the one plotted in the left-hand
sides of the figures below.
The integrals in the sesquilinear form are approximated as described in Remark \ref{rem:quadrature} below (see \eqref{eq:hot2}).

The (squared) error incured by the FEM is measured with the expression
\begin{equation*}
\sum_{\substack{K \in \CT_h \\ K \cap \Omega_{\rm tot} \neq \emptyset}}
\|\widetilde u_{\rm tot}-u_h\|_{H^1(K)}^2,
\end{equation*}
where $\widetilde u_{\rm tot}$ is defined from $u_{\rm tot}$ by extension.
Specifically, for the sound-soft ball, we extend $u_{\rm tot}$
into $B_1(0)$ via the expression for $u_{\rm tot}$ as a Fourier series (which can be evaluated at points in $B_1(0)$). On the other hand, for the penetrable ball,
if $K \in \CT_h$ is an element with barycenter $x_K \in B_1(0)$, $\widetilde u_{\rm tot}|_K$
is defined using the Fourier series of $u_{\rm tot}$ inside $B_1(0)$, whereas 
if $x_K \in B_2(0) \setminus B_1(0)$ then $\widetilde u_{\rm tot}|_K$ is defined using 
the Fourier series of $u_{\rm tot}$ outside $B_1(0)$.

We consider an increasing sequence of wavenumbers ranging from $0.5 \cdot 2\pi$
to $20 \cdot 2\pi$, so that the number of wavelengths in the total field
region where the error is measured ranges from $2$ to $80$. Fixing a polynomial
degree $p$ we then solve the problem for all wavenumber with increasingly refined
meshes. Specifically, we ask {\tt gmsh} to mesh the domain at size $h_{\rm target}$,
where $k^{2p+1}h_{\rm target}^{2p} = C$, and $C$ is chosen so that the relative
error is about 1-2\% for large wavenumbers.

\subsection{Discussion of the numerical results}

\begin{figure}
{
\begin{minipage}{.45\linewidth}
\begin{tikzpicture}
\begin{axis}
[
	width=\linewidth,
	xlabel = {$k$},
	ylabel = {$h$},
	ymode = log,
	xmode = log,
]

\plot table[x expr=2*3.1415*\thisrow{f},y=hmax] {figures/dirichlet/P2/curve.txt};

\plot[domain=3:100] {2*x^(-5./4)};

\SlopeTriangle{0.55}{-0.15}{0.2}{-5./4}{$k^{-5/4}$}{}

\end{axis}
\end{tikzpicture}
\end{minipage}
\begin{minipage}{.45\linewidth}
\begin{tikzpicture}
\begin{axis}
[
	width=\linewidth,
	xlabel = {$k$},
	ylabel = {$\|u-u_h\|_{H^1_k(\Omegat)}/\|u\|_{H^1_k(\Omegat)}$},
	ymode=log,
	xmin=0,
	xmax=130,
	ymin=0.005,
	ymax=0.5
]

\plot table[x expr=2*3.1415*\thisrow{f},y expr=\thisrow{err}/\thisrow{nor}] {figures/dirichlet/P2/curve.txt};

\end{axis}
\end{tikzpicture}
\end{minipage}
\subcaption{Numerical example: sound-soft scattering with $p=2$}
\label{figure_dirichlet_P2}
}
{
\begin{minipage}{.45\linewidth}
\begin{tikzpicture}
\begin{axis}
[
	width=\linewidth,
	xlabel = {$k$},
	ylabel = {$h$},
	ymode = log,
	xmode = log,
]

\plot table[x expr=2*3.1415*\thisrow{f},y=hmax] {figures/dirichlet/P3/curve.txt};

\plot[domain=3:100] {3*x^(-7./6)};

\SlopeTriangle{0.55}{-0.15}{0.2}{-7./6}{$k^{-7/6}$}{}

\end{axis}
\end{tikzpicture}
\end{minipage}
\begin{minipage}{.45\linewidth}
\begin{tikzpicture}
\begin{axis}
[
	width=\linewidth,
	xlabel = {$k$},
	ylabel = {$\|u-u_h\|_{H^1_k(\Omegat)}/\|u\|_{H^1_k(\Omegat)}$},
	ymode=log,
	xmin=0,
	xmax=130,
	ymin=0.005,
	ymax=0.5
]

\plot table[x expr=2*3.1415*\thisrow{f},y expr=\thisrow{err}/\thisrow{nor}] {figures/dirichlet/P3/curve.txt};

\end{axis}
\end{tikzpicture}
\end{minipage}
\subcaption{Numerical example: sound-soft scattering with $p=3$}
\label{figure_dirichlet_P3}
}
{
\begin{minipage}{.45\linewidth}
\begin{tikzpicture}
\begin{axis}
[
	width=\linewidth,
	xlabel = {$k$},
	ylabel = {$h$},
	ymode = log,
	xmode = log,
]

\plot table[x expr=2*3.1415*\thisrow{f},y=hmax] {figures/dirichlet/P4/curve.txt};

\plot[domain=15:100] {6*x^(-9./8)};

\SlopeTriangle{0.5}{-0.15}{0.2}{-9./8}{$k^{-9/8}$}{}

\end{axis}
\end{tikzpicture}
\end{minipage}
\begin{minipage}{.45\linewidth}
\begin{tikzpicture}
\begin{axis}
[
	width=\linewidth,
	xlabel = {$k$},
	ylabel = {$\|u-u_h\|_{H^1_k(\Omegat)}/\|u\|_{H^1_k(\Omegat)}$},
	ymode=log,
	xmin=0,
	xmax=130,
	ymin=0.005,
	ymax=0.5
]

\plot table[x expr=2*3.1415*\thisrow{f},y expr=\thisrow{err}/\thisrow{nor}] {figures/dirichlet/P4/curve.txt};

\end{axis}
\end{tikzpicture}
\end{minipage}
\subcaption{Numerical example: sound-soft scattering with $p=4$}
\label{figure_dirichlet_P4}
}
\caption{Numerical example:~sound-soft scattering}
\end{figure}
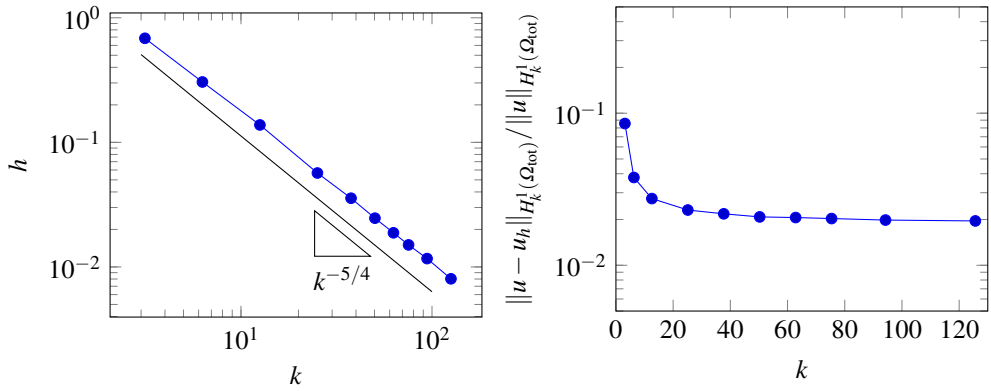
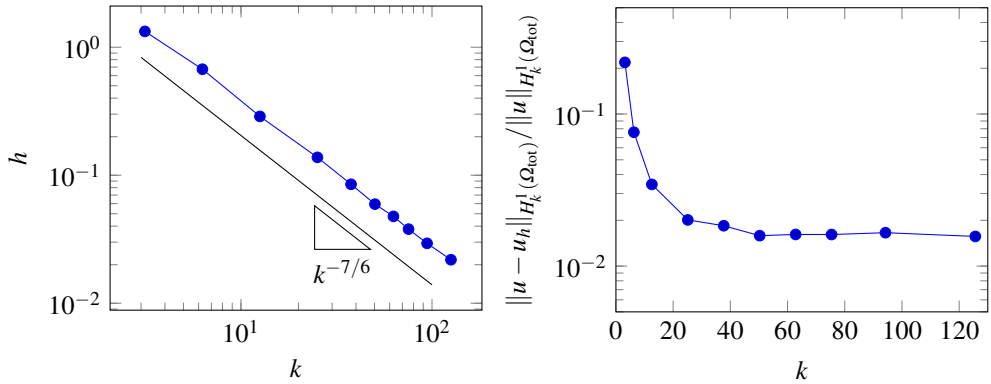
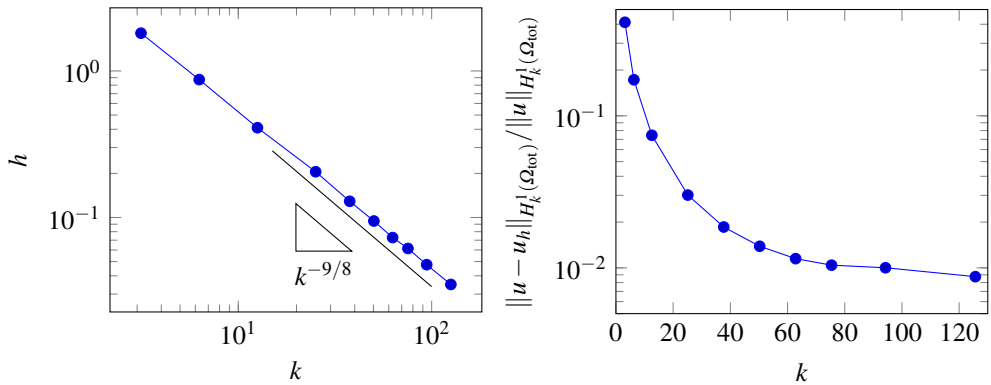

\begin{figure}[h!]
{
\begin{minipage}{.45\linewidth}
\begin{tikzpicture}
\begin{axis}
[
	width=\linewidth,
	xlabel = {$k$},
	ylabel = {$h$},
	ymode = log,
	xmode = log,
]

\plot table[x expr=2*3.1415*\thisrow{f},y=hmax] {figures/transmission/P2/curve.txt};

\plot[domain=3:100] {2*x^(-5./4)};

\SlopeTriangle{0.55}{-0.15}{0.2}{-5./4}{$k^{-5/4}$}{}

\end{axis}
\end{tikzpicture}
\end{minipage}
\begin{minipage}{.45\linewidth}
\begin{tikzpicture}
\begin{axis}
[
	width=\linewidth,
	xlabel = {$k$},
	ylabel = {$\|u-u_h\|_{H^1_k(\Omegat)}/\|u\|_{H^1_k(\Omegat)}$},
	ymode=log,
	xmin=0,
	xmax=130,
	ymin=0.005,
	ymax=0.5
]

\plot table[x expr=2*3.1415*\thisrow{f},y expr=\thisrow{err}/\thisrow{nor}] {figures/transmission/P2/curve.txt};

\end{axis}
\end{tikzpicture}
\end{minipage}
\subcaption{Numerical example: penetrable obstacle with $p=2$}
\label{figure_transmission_P2}
}
{
\begin{minipage}{.45\linewidth}
\begin{tikzpicture}
\begin{axis}
[
	width=\linewidth,
	xlabel = {$k$},
	ylabel = {$h$},
	ymode = log,
	xmode = log,
]

\plot table[x expr=2*3.1415*\thisrow{f},y=hmax] {figures/transmission/P3/curve.txt};

\plot[domain=3:100] {3*x^(-7./6)};

\SlopeTriangle{0.55}{-0.15}{0.2}{-7./6}{$k^{-7/6}$}{}

\end{axis}
\end{tikzpicture}
\end{minipage}
\begin{minipage}{.45\linewidth}
\begin{tikzpicture}
\begin{axis}
[
	width=\linewidth,
	xlabel = {$k$},
	ylabel = {$\|u-u_h\|_{H^1_k(\Omegat)}/\|u\|_{H^1_k(\Omegat)}$},
	ymode=log,
	xmin=0,
	xmax=130,
	ymin=0.005,
	ymax=0.5
]

\plot table[x expr=2*3.1415*\thisrow{f},y expr=\thisrow{err}/\thisrow{nor}] {figures/transmission/P3/curve.txt};

\end{axis}
\end{tikzpicture}
\end{minipage}
\subcaption{Numerical example: penetrable obstacle with $p=3$}
\label{figure_transmission_P3}
}
{
\begin{minipage}{.45\linewidth}
\begin{tikzpicture}
\begin{axis}
[
	width=\linewidth,
	xlabel = {$k$},
	ylabel = {$h$},
	ymode = log,
	xmode = log,
]

\plot table[x expr=2*3.1415*\thisrow{f},y=hmax] {figures/transmission/P4/curve.txt};

\plot[domain=15:100] {6*x^(-9./8)};

\SlopeTriangle{0.5}{-0.15}{0.2}{-9./8}{$k^{-9/8}$}{}

\end{axis}
\end{tikzpicture}
\end{minipage}
\begin{minipage}{.45\linewidth}
\begin{tikzpicture}
\begin{axis}
[
	width=\linewidth,
	xlabel = {$k$},
	ylabel = {$\|u-u_h\|_{H^1_k(\Omega)}/\|u\|_{H^1_k(\Omega)}$},
	ymode=log,
	xmin=0,
	xmax=130,
	ymin=0.005,
	ymax=0.5
]

\plot table[x expr=2*3.1415*\thisrow{f},y expr=\thisrow{err}/\thisrow{nor}] {figures/transmission/P4/curve.txt};

\end{axis}
\end{tikzpicture}
\end{minipage}
\subcaption{Numerical example: penetrable obstacle with $p=4$}
\label{figure_transmission_P4}
}
\caption{Numerical example:~penetrable obstacle}
\end{figure}
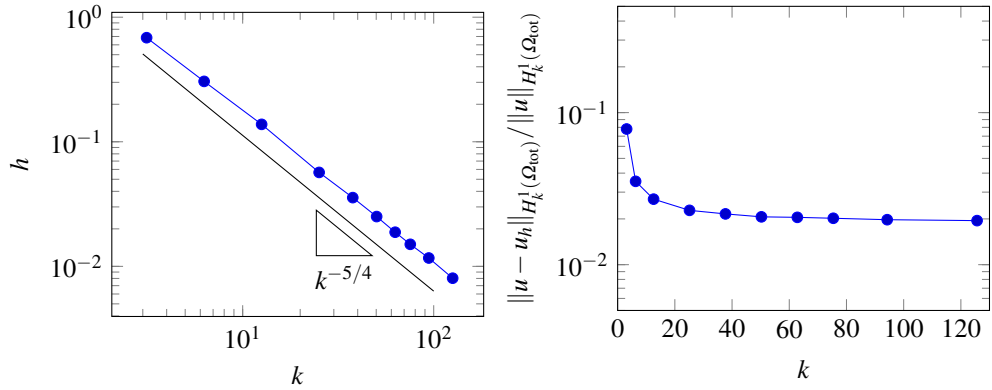
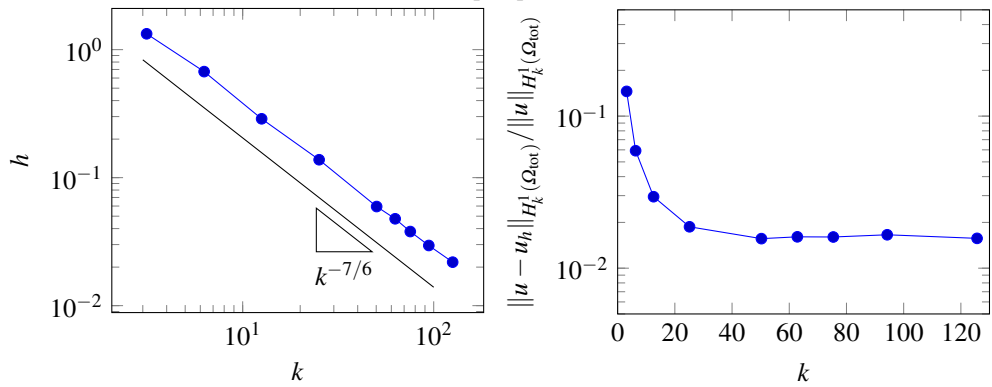
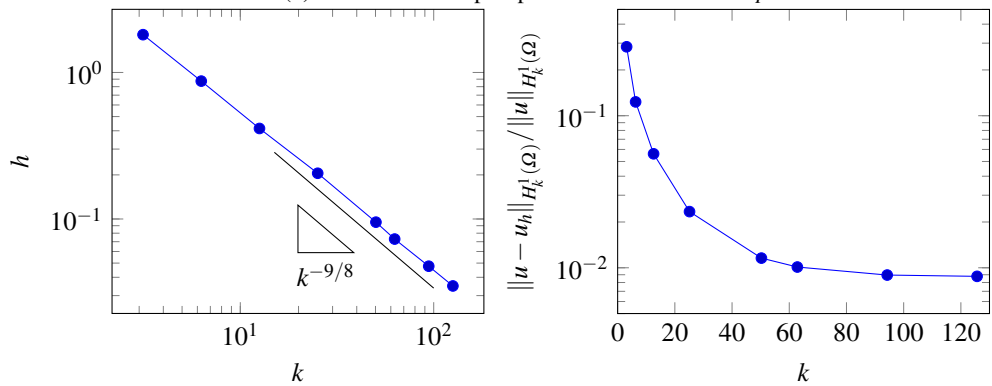

Figures \ref{figure_dirichlet_P2}, \ref{figure_dirichlet_P3}
and \ref{figure_dirichlet_P4} present the results for $p=2$, $3$
and $4$, respectively, for the PML approximation to the sound-soft scattering problem.

Figures \ref{figure_transmission_P2}, \ref{figure_transmission_P3}
and \ref{figure_transmission_P4} present the results for $p=2$, $3$
and $4$, respectively, for the PML approximation to the penetrable-obstacle scattering problem.

As described above, the left-hand sides of the figures plot $h$ as a function of $k$,
illustrating that {\tt gmsh} actually produces a mesh with the desired maximal meshwidth. 
The right-hand sides of the figures plot the relative Galerkin error in $\Omegat$. 

As anticipated from \InformalTheorem \ref{thm:informal1}, the relative error remains uniformly bounded
for all wavenumbers in all cases. We also see that, especially for
higher polynomial degree, the error is larger for the first few wavenumbers
and then stabilizes to a roughly constant value. This is because for the
lowest wavenumbers, the geometric approximation error is not negligible,
and in fact bigger than the pollution error:~indeed, for $k$ fixed, the pollution error on the right-hand side of \eqref{eq:informal1} is $O(h^p)$ whereas the geometric error is $O(h)$ (since $q=1$). When the  wavenumber increases,
the pollution error (which remains constant thanks to our choice of $h$), becomes dominant, as predicted by \InformalTheorem \ref{thm:informal1}.

\section{Strang-type lemma for abstract Helmholtz problems}\label{sec:abstract}

\subsection{Assumptions on the sesquilinear form and finite-dimensional subspace}\label{sec_assumptions}

\subsubsection{Continuity and G\aa rding-type inequality.}

Let $\mathcal{H} \subset H^1(\Omega)$ be a Hilbert space with the norm $\|\cdot\|_{H^1_k(\Omega)}$ \eqref{eq_norm}. (In \S\ref{sec:all_here}, $\mathcal{H}$ is $H^1(\Omega)$ with zero Dirichlet boundary conditions on part of $\partial \Omega$; see \eqref{eq:H1GD} below).

Let $b:\cH\times \cH\to \mathbb{C}$ be a  sesquilinear form that is continuous, i.e., 
\begin{equation}
\label{eq_definition_Ccont}
\sup_{\substack{v \in \mathcal{H} \\ \|v\|_{H^1_k(\Omega)} = 1}}
\sup_{\substack{w \in \mathcal{H} \\ \|w\|_{H^1_k(\Omega)} = 1}}
|b(v,w)| =:\Ccont <\infty
\end{equation}
and satisfies the G\aa rding-type inequality
\begin{equation}
\label{eq_Garding}
\Ccoer \|v\|_{H^1_k(\Omega)}^2
\leq
\Re b(v,v) +  2k^2 \|\operator v\|_{L^2(\Omega)}^2 \quad\tfa v\in \cH,
\end{equation}
for $\Ccoer > 0$ and for some self-adjoint operator $\operator: L^2(\Omega) \to L^2(\Omega)$.

\begin{remark}[The relationship between \eqref{eq_Garding} and the standard G\aa rding inequality]
If $b$ satisfies \eqref{eq_Garding} then $b$ also satisfies a ``standard'' G\aa rding inequality, since 
\begin{equation}
\label{eq_Garding_standard}
\Ccoer \|v\|_{H^1_k(\Omega)}^2
\leq
\Re b(v,v) +  2 k^2 \|\operator\|^2_{L^2(\Omega)\to L^2(\Omega)}\| v\|_{L^2(\Omega)}^2
\end{equation}
for all $v\in \cH$. 
We show in Lemma \ref{lem_GS} below 
(following \cite{GaSp:23}) that if $b$ satisfies a standard G\aa rding inequality along with elliptic regularity, then $b$ satisfies \eqref{eq_Garding} with $\operator$ a smoothing operator; the smoothing property of $\operator$ is then key to obtaining the high-order convergence in the main result (\InformalTheorem \ref{thm:informal1}).
\end{remark}

\subsubsection{The approximate sesquilinear form $b_h$ and approximate right-hand side $\psi_h$.}

Given a finite-dimensional subspace $V_h \subset \mathcal{H}$,
we consider a sesquilinear form $b_h: V_h \times V_h \to \mathbb C$
that is ``close'' to $b$, with this ``closeness" 
characterised by the quantity
\begin{equation}
\label{eq_definition_Cquad}
\Cquad
\eq
\frac{1}{\Ccont}
\sup_{\substack{v_h \in V_h \\ \|v_h\|_{H^1_k(\Omega)} = 1}}
\sup_{\substack{w_h \in V_h \\ \|w_h\|_{H^1_k(\Omega)} = 1}}
|b(v_h,w_h)-b_h(v_h,w_h)|.
\end{equation}
Similarly, given $\psi \in (H^1_k(\Omega))^*$, let $\psi_h: V_h \to \mathbb{C}$ be ``close'' to $\psi$ in the sense that
\begin{equation}\label{eq_definition_Cdrhs}
\Cdrhs := \frac{1}{\Ccont\|\psi\|_{(H^1_k(\Omega))^*}}
\sup_{\substack{v_h \in V_h \\ \|v_h\|_{H^1_k(\Omega)} = 1}}
|\langle \psi - \psi_h,v_h \rangle|.
\end{equation}

\subsubsection{Uniqueness of the solution of the variational problem.}

We assume that the solution of the following variational problem, if it exists, is unique:~given $\psi \in (H^1_k(\Omega))^*$, 
find $v \in \mathcal{H}$ such that
\beq\label{eq_vpb2}
b(v,w) = \langle \psi, w\rangle \quad\tfa w\in \mathcal{H}.
\eeq
The G\aa rding-type inequality \eqref{eq_Garding} implies that $b$ satisfies the standard G\aa rding inequality \eqref{eq_Garding_standard}.
Uniqueness of the solution of \eqref{eq_vpb2} is then equivalent to existence by, e.g., \cite[Theorem 2.32]{Mc:00}, so that 
\eqref{eq_vpb2} has a unique solution $v \in \mathcal{H}$.

\subsubsection{Notation for the adjoint solution operator and adjoint approximability.}

Define the \emph{adjoint solution operator} $\asol: L^2(\Omega)\to \cH$ as the solution of the variational problem
\beq\label{eq_asol}
b\big( v, \asol \phi\big) = (v,\phi)_{L^2(\Omega)} \quad\tfa v \in \mathcal{H}, \,\,\phi \in L^2(\Omega)
\eeq
(the fact that the solution to \eqref{eq_vpb2} is unique implies that $\asol$ is well-defined by, e.g., \cite[Theorem 2.27]{Mc:00}).
We then define
\begin{equation}\label{eq_definition_Cstab}
\Cstab
\eq k \N{\asol \operator}_{L^2(\Omega)\to H^1_k(\Omega)};
\end{equation}
i.e,. $\Cstab$ is the norm of the solution operator of the adjoint problem with a $\operator$ on the right-hand side.
The reason for including the factor of $k$ in \eqref{eq_definition_Cstab} is that $\Cstab$ is then dimensionless.

Denote the orthogonal projector to $V_h$ in $\|{\cdot}\|_{H^1_k(\Omega)}$ by
\begin{equation}
\label{eq_definition_pih}
\pi_h w \eq \arg \min_{w_h \in V_h} \|w-w_h\|_{H^1_k(\Omega)} \quad\tfa w\in \cH.
\end{equation}
The following quantity measures how well solutions of the adjoint problem with a $\operator$ on the right-hand side are approximated in $V_h$:
\begin{equation}
\label{eq_definition_Caprx}
\Caprx
\eq
k\N{ (I-\pi_h)\asol \operator}_{L^2(\Omega)\to H^1_k(\Omega)};
\end{equation}
similar to with $\Cstab$, the factor of $k$ in \eqref{eq_definition_Caprx} means that $\Caprx$ is dimensionless.

\subsubsection{The modified sesquilinear form and its solution operator.}
Define 
\beq\label{eq_btilde}
\newb (v,w) := b(v,w) + 2k^2\big( \operator^2 v, w\big)_{L^2(\Omega)},
\eeq
so that $\newb $ is coercive on $\cH$ by \eqref{eq_Garding} (with coercivity constant $\Ccoer$).
Let 
\begin{equation}\label{eq_tb_cty}
\sup_{\substack{v \in \mathcal{H} \\ \|v\|_{H^1_k(\Omega)} = 1}}
\sup_{\substack{w \in \mathcal{H} \\ \|w\|_{H^1_k(\Omega)} = 1}}
\big|\newb (v,w)\big| =:\Ccontt \leq \Ccont + 2\|\operator^2\|_{L^2(\Omega)\to L^2(\Omega)}.
\end{equation}
Let 
$\tsol:L^2(\Omega)\to \cH$ be defined as the solution of the variational problem
\beq\label{eq_tsol}
\newb \big( \tsol \phi,w\big) = (\phi,w)_{L^2(\Omega)} \quad\tfa w \in \mathcal{H}, \,\,\phi\in L^2(\Omega).
\eeq
Similar to \eqref{eq_definition_Caprx},
\beq\label{eq_definition_Caprxop}
\Caprxop
:=
k \big\|(I-\pi_h)\tsol \operator\big\|_{L^2(\Omega)\to H^1_k(\Omega)}
\eeq
measures how well solutions of the variational problem involving $\newb $ and with a $\operator$ on the right-hand side are approximated in $V_h$.

\subsection{The main abstract result}

The following result is an abstract version of the elliptic-projection argument from \cite{GaSp:23}, extended to allow a variational crime. 
This result involves the quantities $\Ccont, \Ccoer, \Cquad,\Cdrhs, \Cstab,\Caprx, \Ccontt, \Caprxop$ and the operator $\operator$, all defined in \S\ref{sec_assumptions}.

\begin{theorem}[Abstract elliptic-projection-type argument with variational crime]
\label{thm_abs1}
Under the assumptions in \S\ref{sec_assumptions}, 
let $u\in \cH$ be the unique solution of 
the variational problem
\beq\label{eq_vpb}
b(\PMLsol,w) = \langle \psi, w\rangle \quad\tfa w\in \mathcal{H}.
\eeq
Let 
\beqs
A:= 1 - 2\frac{(\Ccontt)^2}{\Ccoer}  
\Caprx \Caprxop
\eeqs
and 
\beqs
B:= 
A\big( \Ccoer - \Ccont \Cquad\big) - \sqrt{\frac{2}{\Ccoer}}\Ccont \Ccontt 
\Cstab \Cquad.
\eeqs
If $B>0$, then there exists a unique $u_h \in V_h$ such that
\beq\label{eq_vpbh}
b_h(\PMLsol_h,w_h) = \langle \psi_h, w_h\rangle \quad\tfa w_h \in V_h,
\eeq
and the error $u-u_h$ satisfies the bound 
\begin{align}\nonumber
B\N{u-u_h}_{H^1_k(\Omega)} \leq &\bigg[
\Ccont + \sqrt{\frac{2}{\Ccoer}}(\Ccontt)^2 \Caprx
\bigg]
\N{(I- \pi_h) u}_{H^1_k(\Omega)} \\ 
&\qquad
+ \bigg[
\Ccont + \sqrt{\frac{2}{\Ccoer}}\Ccont \Ccontt\Cstab
\bigg]
\Big( \Cquad \N{u}_{H^1_k(\Omega)} + \Cdrhs\N{\psi}_{(H^1_k(\Omega))^*}
\Big).\label{eq:abs_bound1}
\end{align}
\end{theorem}

We make three immediate remarks.
\begin{enumerate} 
\item The bound \eqref{eq:abs_bound1} has a similar structure to the first Strang lemma for coercive problems, namely the Galerkin error is bounded by the sum of (i) the best approximation error, (ii) the error in the sesquilinear form, and (iii) the error in the right-hand side;
see \cite{St:72a}, \cite[Theorem 26.1]{Ci:91}, \cite[Theorem 4.1.1]{Ci:02}.
\item 
If $\Cquad,\Cdrhs=0$ and the approximation factors $\Caprx,\Caprxop\to0$, then $A\to 1$, $B\to \Ccoer$, and the bound \eqref{eq:abs_bound1} approaches 
\beqs
\N{u-u_h}_{H^1_k(\Omega)}\leq \frac{\Ccont}{\Ccoer}\N{(I- \pi_h) u}_{H^1_k(\Omega)};
\eeqs
observe that this is the bound given by C\'ea's lemma for a coercive problem with continuity constant $\Ccont$ and coercivity constant $\Ccoer$.
\item
Theorem \ref{thm_abs1} is, in fact, a slight refinement of the arguments of \cite{GaSp:23}, in that both $\Cstab$ \eqref{eq_definition_Cstab} and the ``adjoint approximation factor" $\Caprx$ \eqref{eq_definition_Caprx} involve $\asol$ applied to $\operator$ 
 (which we see in Lemma \ref{lem_GS} below is a smoothing operator )
whereas in the arguments of  \cite{GaSp:23} $\asol$ does not only appear acting on $\operator$. In particular, the arguments of \cite{GaSp:23} use the more-standard adjoint approximation factor 
$k\N{ (I-\pi_h)\asol}_{L^2(\Omega)\to H^1_k(\Omega)}$ introduced in \cite{Sa:06} coming out of the work of \cite{Sc:74}. 
For a discussion of how the ideas behind Theorem \ref{thm_abs1}/the results of \cite{GaSp:23} are related to previous duality arguments in the literature, see \cite[Remark 2.2]{GaSp:23}.
\end{enumerate}

\subsection{Outline of how Theorem \ref{thm_abs1} gives \InformalTheorem \ref{thm:informal1}. }
Assuming that $\Ccont,\Ccontt$, and $\Ccoer$ are independent of $k$ (as we see below they are for Helmholtz problems), we see that 
Theorem \ref{thm_abs1} shows that if 
\beq\label{eq_simple0}
\Caprx 
\Caprxop
 \quad\tand\quad \Cstab \Cquad \quad \text{ are sufficiently small }
\eeq
then
\beq\label{eq_simple1}
\N{u-u_h}_{H^1_k(\Omega)}\leq C (1 + \Caprx) \N{(I- \pi_h) u}_{H^1_k(\Omega)} + C (1+ \Cstab)\Big( \Cquad \N{u}_{H^1_k(\Omega)} + \Cdrhs\N{\psi}_{(H^1_k(\Omega))^*}\Big).
\eeq
Continuity of $b$ \eqref{eq_definition_Ccont} implies that $\N{\psi}_{(H^1_k(\Omega))^*}\leq \Ccont \N{u}_{H^1_k(\Omega)}$, and thus \eqref{eq_simple1} implies that 
\beq\label{eq_simple2}
\N{u-u_h}_{H^1_k(\Omega)}\leq C (1 + \Caprx) \N{(I- \pi_h) u}_{H^1_k(\Omega)} + C (1+ \Cstab)\big( \Cquad + \Cdrhs\big)\N{u}_{H^1_k(\Omega)}.
\eeq

Section \ref{sec:all_here} shows how 
for the particular Helmholtz problem considered in Section \ref{sec:intro}, \eqref{eq_simple0} and \eqref{eq_simple2} become \eqref{eq:condition1} and \eqref{eq:informal1}, 
respectively; indeed, this is via the bounds 
$\Caprx\leq C k^\alpha(kh)^p$, $
\Caprxop
\leq C (kh)^p$, 
$\Cstab\leq C k^{\alpha}$,
and $\Cquad, \Cdrhs \leq C h^q$ (for a domain $\Omega$ with characteristic length scale $\sim 1$).

\subsection{Proof of Theorem \ref{thm_abs1}}

Theorem \ref{thm_abs1} comes from combining Lemmas \ref{lem:abs1} and \ref{lem:abs2} below.
These lemmas involve the following projection operator; let $\tpi_h: \cH \to V_h$ be defined by 
\beqs
\newb \big(v_h , \tpi_h w \big)  = \newb  (v_h, w) \quad \text{ for all } v_h \in V_h, w \in \cH;
\eeqs
i.e., 
\beq\label{eq_Gogt}
\newb \big(v_h , (I-\tpi_h) w \big)  = 0\quad \text{ for all } v_h \in V_h, w \in \cH.
\eeq
We immediately note that, since $\newb $ is continuous and coercive, 
C\'ea's lemma implies that
\beq\label{eq_tpi_pi}
\big\| (I- \tpi_h )v\big\|_{H^1_k(\Omega)} \leq \frac{\Ccontt}{\Ccoer} \big\| (I- \pi_h)v \big\|_{ H^1_k(\Omega)} \tfa v \in H^1_k(\Omega),
\eeq
and 
\beq\label{eq_norm_tpi}
\big\| \tpi_h v\big\|_{H^1_k(\Omega)} \leq \frac{\Ccontt}{\Ccoer} \N{v}_{H^1_k(\Omega)} \quad \tfa v \in \cH,
\eeq
since 
\beqs
\Ccoer \N{\tpi_h v}^2_{H^1_k(\Omega)} \leq \newb \big( \tpi_h v, \tpi_h v\big) = \newb \big( \tpi_h v,  v\big) \leq \Ccontt \big\| \tpi_h v\|_{H^1_k(\Omega)} \N{v}_{H^1_k(\Omega)}.
\eeqs

In both the statement and proof of the next lemma, we drop the $(\Omega)$ in norms,
i.e., we write, e.g., $\|\operator\|_{L^2\to L^2}$ instead of
$\|\operator \|_{L^2(\Omega)\to L^2(\Omega)}$ to keep expressions compact.

\begin{lemma}\label{lem:abs1}
Let 
\beq\label{eq_definition_Caprxt}
\Caprxt:= k^2 \big\| \operator (I- \tpi_h) \asol \operator\big\|_{L^2\to L^2}.
\eeq
If the solution $u_h$ to \eqref{eq_vpbh} exists, then 
\begin{align}\nonumber
&\bigg[ 
\big( 1- 2\Caprxt 
\big) \Big( \Ccoer -\Ccont \Cquad\Big) - \sqrt{\frac{2}{\Ccoer}}\Ccont \Ccontt 
\Cstab \Cquad
\bigg]
\N{u-u_h}_{H^1_k(\Omega)}
\\ \nonumber
&
\leq 
\bigg[
\big( 1- 2\Caprxt 
\big)\Ccont + \sqrt{\frac{2}{\Ccoer}}(\Ccontt)^2 \Caprx 
\bigg]
\N{(I-\pi_h)u}_{H^1_k(\Omega)} 
\\
&
\qquad+ \bigg[
\big( 1- 2\Caprxt 
\big)\Ccont +\sqrt{\frac{2}{\Ccoer}}\Ccont \Ccontt 
\Cstab 
\bigg] 
\Big( \Cquad \N{u}_{H^1_k(\Omega)} + \Cdrhs\N{\psi}_{(H^1_k(\Omega))^*}\Big).\label{eq:abs_bound2}
\end{align}
\end{lemma}

Note that \eqref{eq:abs_bound2} has the same structure as \eqref{eq:abs_bound1}, 
with $1-2\Caprxt$ in \eqref{eq:abs_bound2} replaced by $A$ in \eqref{eq:abs_bound1}. 
The following lemma does this replacement, by bounding $\Caprxt$ by $\Caprxop\Caprx$
(see \eqref{eq_extra1} below). 

\begin{lemma}\label{lem:abs2}
For all $v\in \cH$,
\beqs
k\big\| \operator ( I- \tpi_h) v\big\|_{L^2(\Omega)} \leq \Ccontt 
\Caprxop
 \big\| (I-\tpi_h)v\big\|_{H^1_k(\Omega)}.
\eeqs
\end{lemma}

\bpf[Proof of Theorem \ref{thm_abs1} using Lemmas \ref{lem:abs1} and \ref{lem:abs2}]
Once the bound \eqref{eq:abs_bound1} is established under the assumption that $u_h$ exists, applying this bound with $\psi=0$ and using uniqueness of the continuous problem \eqref{eq_vpb} implies that $u_h$, if it exists, is unique. Existence of $u_h$ then follows since $u_h$ is the solution of a finite-dimensional linear system, for which uniqueness is equivalent to existence.

It is therefore sufficient to prove \eqref{eq:abs_bound1} under the assumption that $u_h$ exists.
By the definition of $\Caprxt$ \eqref{eq_definition_Caprxt}, Lemma \ref{lem:abs2}, 
\eqref{eq_tpi_pi}, and the definition of $\Caprx$ \eqref{eq_definition_Caprx},
\beq\label{eq_extra1}
\Caprxt \leq 
k \Ccontt \Caprx \big\| (I-\pi_h^+)\asol \operator\big\|_{L^2(\Omega)\to H^1_k(\Omega)}
\leq
\frac{(\Ccontt)^2}{\Ccoer} \Caprxop \Caprx.
\eeq
Therefore
\beqs
A:= 1 - 2\frac{(\Ccontt)^2}{\Ccoer}  \Caprxop \Caprx\leq 1-2\Caprxt \leq 1
\eeqs
and using this in \eqref{eq:abs_bound2} gives the result \eqref{eq:abs_bound1}.
\epf

\

We now prove Lemma \ref{lem:abs2} using a duality argument involving the sesquilinear form $\newb $.

\

\bpf[Proof of Lemma \ref{lem:abs2}]
By the definition of $\tsol$ \eqref{eq_tsol}, the Galerkin orthogonality \eqref{eq_Gogt}, and continuity of $\newb $ \eqref{eq_tb_cty}
\begin{align*}
\big\| \operator (I-\tpi_h) v \big\|^2_{L^2(\Omega)}
&= \big ( \operator^2 (I-\tpi_h)v , (I-\tpi_h)v \big )_{L^2(\Omega)}\\
&= \newb \big(\tsol\operator^2 (I-\tpi_h)v , (I-\tpi_h)v\big)\\
&= \newb \big((I-\pi_h)\tsol\operator^2 (I-\tpi_h)v , (I-\tpi_h)v\big)\\
&\leq \Ccontt \big\| 
(I-\pi_h) \tsol \operator
\big\|_{L^2(\Omega)\to H^1_k(\Omega)} 
\big\|
\operator (I-\tpi_h) v
\big\|_{L^2(\Omega)}
\big\|
 (I-\tpi_h) v
 \big\|_{H^1_k(\Omega)},
\end{align*}
and the result follows by the definition of $\Caprxop$ \eqref{eq_definition_Caprxop}.
\epf

\

To complete the proof of Theorem \ref{thm_abs1}, we therefore need to prove Lemma \ref{lem:abs1}.
To do this, we start with a simple result controlling the variational crime error.

\begin{lemma}[``Galerkin orthogonality'' in the sesquilinear form $b(\cdot,\cdot)$]
If $\PMLsol \in \mathcal{H}$ 
is the solution of the variational problem \eqref{eq_vpb} 
and $\PMLsol_h \in V_h$ is the solution of the variational problem
\eqref{eq_vpbh}, then, for all $w_h\in V_h$,
\begin{equation}
\label{eq_quasi_orthogonality}
|b(\PMLsol-\PMLsol_h,w_h)|
\leq
\Ccont \Big( \Cquad \|\PMLsol-\PMLsol_h\|_{H^1_k(\Omega)}+\Cquad \|\PMLsol\|_{H^1_k(\Omega)}
+ \Cdrhs \|\psi\|_{(H^1_k(\Omega))^*}
\Big)\|w_h\|_{H^1_k(\Omega)}.
\end{equation}
\end{lemma}

\begin{proof}
By  \eqref{eq_vpb} and \eqref{eq_vpbh},
\begin{align*}
b(\PMLsol-\PMLsol_h,w_h) = b(\PMLsol,w_h)-b(\PMLsol_h,w_h) & = \langle \psi, w_h\rangle  - b_h(\PMLsol_h,w_h)+ \big( b_h(\PMLsol_h,w_h) - b(\PMLsol_h,w_h)\big)\\
&= \langle \psi-\psi_h,w_h \rangle + \big( b_h(\PMLsol_h,w_h) - b(\PMLsol_h,w_h)\big).
\end{align*}
Then, by the definitions of $\Cdrhs$ \eqref{eq_definition_Cdrhs} and $\Cquad$ \eqref{eq_definition_Cquad},
\begin{equation*}
|b(\PMLsol-\PMLsol_h,w_h)| \leq \Ccont\Big( \Cdrhs \|\psi\|_{(H^1_k(\Omega))^*} + \Cquad 
\|\PMLsol_h\|_{H^1_k(\Omega)}\Big)\|w_h\|_{H^1_k(\Omega)}.
\end{equation*}
The result \eqref{eq_quasi_orthogonality} then follows from the triangle inequality.
\end{proof}

\begin{lemma}[Useful quadratic inequality]
\label{lemma_quadratic}
Let $a>0$, $b,c \geq 0$ and $x \geq 0$. If
\begin{equation}\label{eq_quadratic1}
ax^2 \leq bx + c^2,
\end{equation}
then 
\begin{equation}
\label{eq_quadratic}
ax \leq b+c\sqrt{a}.
\end{equation}
\end{lemma}

\begin{proof}
Multiplying the inequality \eqref{eq_quadratic1} by $a>0$, we find that
\beqs
\left( a x - \frac{b}{2}\right)^2 - \frac{b^2}{4} \leq a c^2, \quad \text{ and thus }\quad
ax \leq \frac{b}{2} + \sqrt{\frac{b^2}{4}+ac^2}.
\end{equation*}
The result \eqref{eq_quadratic} follows since 
\begin{equation*}
\sqrt{\frac{b^2}{4}+ac^2}
\leq
\sqrt{\frac{b^2}{4}}+\sqrt{ac^2}
\leq
\frac{b}{2} + c\sqrt{a}.
\end{equation*}
\end{proof}

\

\bpf[Proof of Lemma \ref{lem:abs1}]
Let $e := u-u_h$. By the G\aa rding-type inequality \eqref{eq_Garding}, \eqref{eq_quasi_orthogonality}, and continuity of $b$ \eqref{eq_definition_Ccont},
\begin{align*}
\Ccoer \N{ e}^2_{H^1_k} &\leq \Re b(e,e) + 2k^2 \N{\operator e}^2_{L^2} \\
& = \Re b\big(e, (I-\pi_h)u\big) + 2k^2 \N{\operator e}^2_{L^2} - \Re b(e, u_h - \pi_h u) \\
&\leq \Ccont \N{e}_{H^1_k} \N{(I-\pi_h) u}_{H^1_k} + 2k^2 \N{\operator e}^2_{L^2} 
 \\
& \hspace{3cm}+ \Ccont \Big( \Cquad \N{e}_{H^1_k}+ \Cquad \N{u}_{H^1_k} + \Cdrhs \N{\psi}_{(H^1_k)^*}\Big) \N{e}_{H^1_k},
\end{align*}
where to obtain the last inequality we have used (i) that $u_h -\pi_h u = \pi_h e$ and (ii) $\N{\pi_h e}_{H^1_k}\leq \N{e}_{H^1_k}$ since $\pi_h$ is a projection 
\eqref{eq_definition_pih}. Rearranging the last displayed equation, we obtain that
\begin{align*}
\big(\Ccoer - \Ccont \Cquad\big) \N{e}^2_{H^1_k} \leq \Ccont \N{e}_{H^1_k} \Big( \N{(I-\pi_h)u}_{H^1_k} &+ \Cquad\N{u}_{H^1_k} + \Cdrhs\N{\psi}_{(H^1_k)^*}\Big) 
+ 2k^2 \N{\operator e}^2_{L^2}.
\end{align*}
Then, by Lemma \ref{lemma_quadratic} and the fact that $\sqrt{\Ccoer -\Ccont \Cquad}\leq \sqrt{\Ccoer}$,
\begin{align}\label{eq_cold0}
\big(\Ccoer - \Ccont \Cquad\big) \N{e}_{H^1_k} \leq \Ccont \Big( \N{(I-\pi_h)u}_{H^1_k} + \Cquad\N{u}_{H^1_k} + \Cdrhs\N{\psi}_{(H^1_k)^*}\Big) 
+ \sqrt{2\Ccoer}k \N{\operator e}_{L^2}.
\end{align}
We now use a duality argument to bound $\| \operator e\|_{L^2}$. By the self-adjointness of $\operator$, the definition of $\asol$ \eqref{eq_asol}, \eqref{eq_quasi_orthogonality}, and \eqref{eq_norm_tpi},
\begin{align}\nonumber
\N{\operator e}^2_{L^2} &= b( e ,\asol \operator^2 e) \\ \nonumber
& = b\big( e, (I-\tpi_h) \asol \operator^2 e\big) + b\big( e, \tpi_h \asol \operator^2 e\big) \\ \nonumber
& \leq b\big( e, (I-\tpi_h) \asol \operator^2 e\big)  
\\
&\quad+ \Ccont \Big( \Cquad\N{e}_{H^1_k} + \Cquad\N{u}_{H^1_k} + \Cdrhs\N{\psi}_{(H^1_k)^*}\Big)\frac{\Ccontt}{\Ccoer} \big\|\asol \operator\|_{L^2\to H^1_k} \N{\operator e}_{L^2}.\label{eq_cold1}
\end{align}
Now, by the definition of $\newb $ \eqref{eq_btilde}, the Galerkin orthogonality \eqref{eq_Gogt}, the self-adjointness of $\operator$, continuity of $\newb $ \eqref{eq_tb_cty}, 
\eqref{eq_tpi_pi}, and the definitions of $\Caprx$ \eqref{eq_definition_Caprx} and $\Caprxt$ \eqref{eq_definition_Caprxt},
\begin{align}\nonumber
b\big( e, (I-\tpi_h) \asol \operator^2 e\big)
&= \newb \big( (I-\pi_h) u , (I-\tpi_h)\asol \operator^2 e\big) - 2k^2 \big( \operator^2 e, (I-\tpi_h) \asol \operator^2 e\big)_{L^2}\\ \nonumber
&= \newb \big( (I-\pi_h) u , (I-\tpi_h)\asol \operator^2 e\big) -  2 k^2 \big( \operator e, \operator(I-\tpi_h) \asol \operator^2 e\big)_{L^2}\\
&\leq \Ccontt \N{(I-\pi_h)u}_{H^1_k} \frac{\Ccontt}{\Ccoer} k^{-1}\Caprx 
\N{\operator e}_{L^2} + 2\Caprxt 
 \N{\operator e}_{L^2}^2.\label{eq_cold2}
\end{align}
Combining \eqref{eq_cold1} and \eqref{eq_cold2}, multiplying by $k$, dividing by $\|\operator e\|_{L^2}$, and using the definition of $\Cstab$ \eqref{eq_definition_Cstab}, we obtain that 
\begin{align}
\Big( 1 - 2\Caprxt 
\Big) k\N{\operator e}_{L^2} &\leq \frac{(\Ccontt)^2}{\Ccoer}  \Caprx 
\N{(I-\pi_h)u}_{H^1_k} 
+ \frac{\Ccont \Ccontt}{\Ccoer} 
\Cstab\Big( \Cquad \N{e}_{H^1_k} + \Cquad\N{u}_{H^1_k} + \Cdrhs\N{\psi}_{(H^1_k)^*}\Big).\label{eq_cold3}
\end{align}
Multiplying \eqref{eq_cold0} by $( 1 - 2\Caprxt)$ and inputting \eqref{eq_cold3}, we obtain
\begin{align*}
&\Big( 1 - 2\Caprxt
\Big)\big(\Ccoer - \Ccont \Cquad\big) \N{e}_{H^1_k}
\\
& \hspace{0.75cm}\leq 
\Big( 1 - 2\Caprxt
\Big)
\Ccont \Big( \N{(I-\pi_h)u}_{H^1_k} + \Cquad\N{u}_{H^1_k} + \Cdrhs\N{\psi}_{(H^1_k)^*}\Big) 
+
\sqrt{\frac{2}{\Ccoer}}(\Ccontt)^2
 \Caprx
\N{(I-\pi_h)u}_{H^1_k}
\\& \hspace{3.5cm}
+ \sqrt{\frac{2}{\Ccoer}}\Ccont \Ccontt
\Cstab\Big( \Cquad \N{e}_{H^1_k} + \Cquad\N{u}_{H^1_k} + \Cdrhs\N{\psi}_{(H^1_k)^*}\Big),
\end{align*}
which then rearranges to the result \eqref{eq:abs_bound2}.
\epf

\section{Applying the abstract theory to Helmholtz problems solved with curved finite elements}
\label{sec:all_here}

\subsection{The Helmholtz problem with piecewise-smooth coefficients}\label{sec_4.1}

Let $\Omega \subset \mathbb R^d$, $d=2,3$, be a bounded Lipschitz domain with characteristic length scale $L$. The boundary $\partial \Omega$ of $\Omega$ is partitioned into two
disjoint relatively open Lipschitz subsets $\GD$ and $\GN$ (where the subscripts ``${\rm D}$'' and ``${\rm N}$'' stand for ``Dirichlet'' and ``Neumann"). 
We highlight that this set-up means that 
the boundaries of $\GD$ and $\GN$ do not touch.
Let $\cH$ be the Hilbert space
\beq\label{eq:H1GD}
H^1_{\GD}(\Omega):=\big\{ v\in H^1(\Omega) : v= 0 \,\ton \GD\big\}.
\eeq
The domain $\Omega$ is partitioned into disjoint
open Lipschitz subsets $Q \in \Part$, and we assume that the restrictions of the PDE
coefficients to each $Q$ are smooth. Specifically, for each $Q \in \Part$, we consider
two functions
$\mu_Q: \mathbb R^d \to \mathbb C$ and
$A_Q: \mathbb R^d \to \mathbb C^{d \times d}$ and the sesquilinear form then reads
\begin{align}\nonumber
b(v,w)
&\eq
\sum_{Q \in \Part}
\Big \{-k^2 (\mu_Qv,w)_{L^2(Q)}
+
(A_Q\nabla v,\nabla w)_{L^2(Q)}
\Big \}\\
&=
-k^2(\mu v,w)_{L^2(\Omega)}
+(A\nabla v,\nabla w)_{L^2(\Omega)}
\quad
\tfa v,w \in H^1_{\GD}(\Omega),
\label{eq_actual_b}
\end{align}
where $\mu|_Q := \mu_Q$, $A|_Q :=A_Q$.
To define the sesquilinear form $b$, the values of $\mu_Q$ and $A_Q$ are not needed outside $Q$. However,
assuming that they are also defined in a neighborhood of $Q$ is useful when defining the discrete bilinear form (involving geometric error).
We note that in the case of piecewise constant coefficients or coefficients
given by analytical formula (e.g., in a PML) such extensions are automatically defined.

\begin{assumption}[Strong ellipticity of the Helmholtz operator]\label{ass_strong_ellipticity}
There exists $c>0$ such that
\beqs
\Re \sum_{i=1}^d \sum_{j=1}^d A_{ij}(x) \xi_j \overline{\xi_i} \geq c |\xi|^2 \quad\tfa x \in \Omega \tand \xi\in \mathbb{R}^d.
\eeqs
\end{assumption}

Assumption \ref{ass_strong_ellipticity} implies that the Helmholtz operator is strongly elliptic in the sense of \cite[Equation 4.7]{Mc:00} by, e.g., \cite[Page 122]{Mc:00}.

\begin{assumption}[Regularity assumption on $A,\mu$, and $\partial Q$]\label{ass_regularity}
For some $\ell \in \mathbb{Z}^+$,
$A_Q, \mu_Q \in C^{\ell-1 ,1}$ and $\partial Q \in C^{\ell,1}$ for all $Q\in \mathcal{Q}$.
\end{assumption}

Note that the requirement in Assumption \ref{ass_regularity} that $\partial Q \in C^{\ell,1}$ excludes ``triple points" where three of the $Q$ regions meet.
Assumption \ref{ass_regularity} implies that the Helmholtz operator satisfies the standard elliptic-regularity shift (see, e.g., \cite[Theorem 4.20]{Mc:00}). 
Thanks to 
\cite{GaSp:23}
this shift property can be used to bound the quantities $\Caprx$ \eqref{eq_definition_Caprx} and $\Caprxop$ \eqref{eq_definition_Caprxop}, respectively, appearing in Theorem \ref{thm_abs1}; see \S\ref{sec:aprxop} below.

We additionally assume that 
\beq\label{eq:RHS}
\langle \psi, v \rangle =  (\data,v)_{L^2(\Omega)}
\eeq
with $\data= \data_Q$ on $Q$ and $\data_Q: \mathbb R^d \to \mathbb C$ with $\data_Q\in H^1(\mathbb{R}^d)$.

Given the partition $\Part$ of $\Omega$, we use the notation that
$v\in H^j(\mathcal{Q})$ iff 
 $v|_{Q} \in H^j(Q)$ for all $Q\in \mathcal{Q}$, and we let 
\beq\label{eq_norm_Part}
\|v\|_{H^m(\mathcal{Q})}^2:= \sum_{Q\in \mathcal{Q}}\|v\|_{H^m(Q)}^2
\eeq
where 
\beqs
\N{u}^2_{H^m(D)}:= \sum_{|\alpha|\leq m} L^{2(m-|\alpha|)} \N{\partial^\alpha u}^2_{L^2(D)}
\quad \tfor m\in\mathbb{Z}^+,
\eeqs
where $L$ is the characteristic length scale of $\Omega$; these factors of $L$ are included to make each term in the norm have the same dimension.

\subsection{Curved finite-element setting}\label{sec_4.2}

Here, we introduce a finite element space on curved simplicial elements.
For such spaces, there are the following two key considerations. 
\begin{enumerate}
\item[(a)] On the one hand,
the mesh elements should be sufficiently curved so as to correctly approximate
the geometry of the boundary value problem (i.e. $\Omega$, $\GD$, $\GN$ and $\Part$).
\item[(b)] On the other hand, the elements should not be too distorted so as to preserve
the approximation properties of the finite element space. 
\end{enumerate}
These two aspects have
been thoroughly investigated in the finite-element literature 
(see e.g. \cite{Ci:02}, \cite{ciarlet_raviart_1972a} and \cite{lenoir_1986a}),
and this section summarises the key assumptions needed in our analysis.

We consider a non-conforming mesh $\CT_h$ that does not exactly cover $\Omega$
nor the partition $\Part$. The (closed) elements $K \in \CT_h$ are obtained by
mapping a single (closed) reference simplex $\hK$ through bilipschitz maps $\MAP_K: \hK \to K$.
The maps $\MAP_K$ could be affine (leading to straight elements) but we more generally
allow polynomials mappings of degree $q \geq 1$ leading to curved elements. We classicaly
assume that the mesh is ``matching'', meaning that
$\overset{\circ}{K} \cap \overset{\circ}{K'} = \emptyset$ for all $K,K' \in \CT_h$,
and $K \cap K'$ is either empty, or corresponds to a single vertex, a single edge
or a single face of the reference simplex $\hK$ mapped through $\MAP_K$
(or equivalently, $\MAP_{K'}$).

Let 
\beqs
\Omegah \eq \operatorname{Int}\Big(\bigcup_{K \in \CT_h} K\Big);
\eeqs
i.e., $\Omegah$ is the domain
covered by $\CT_h$. In general, $\Omegah \neq \Omega$, but we expect
that $\CT_h$ is ``close'' to matching the boundary and interfaces in the PDE problem.
To encode this, we assume that for each $K \in \CT_h$, there exists
a bilipschitz mapping $\Psi^h_K: K \to \mathbb R^d$, such that if $\tK \eq \Psi_K^h(K)$,
then
\begin{enumerate}
\item[(i)] there exists $Q \in \Part$ such that $\tK \subset \overline{Q}$,
\item[(ii)] $\bigcup_{K \in \CT_h} \tK = \overline{\Omega}$,
\item[(iii)] letting $\Psi^h|_K := \Psi^h_K$ for all $K \in \CT_h$ defines a bilipschitz
mapping $\Psi^h: \overline{\Omegah} \to \overline{\Omega}$.
\end{enumerate}
Let 
\beqs
\Phi^h_K \eq (\Psi^h_K)^{-1} \quad\tand\quad 
\Phi^h \eq (\Psi^h)^{-1}.
\eeqs 
Note then that $\Phi^h|_{\tK} = \Phi^h_K$
for all $K \in \CT_h$. Figure \ref{figure_example} illustrates this set-up in a particular example.
We  let 
\beqs
\GDh \eq \Phi^h(\GD), \quad \GNh \eq \Phi^h(\GN), \quad Q^h \eq \Phi^h(Q)\, \tfor Q \in \Part;
\eeqs
and 
\beqs
\Parth \eq \{ Q^h \}_{Q \in \Part}.
\eeqs
Notice that
if $\Psi^h = I$, then $\Omegah = \Omega$, $\GDh = \GD$,
$\GNh = \GN$ and $Q^h = Q$ for all $Q \in \CT_h$, so that the mesh is
actually conforming. We therefore expect each map $\Psi_K^h$ to be close
to the identity in a suitable sense (made precise in Assumption \ref{assumption_curved_fem} below). Without loss of
generality we assume that $\det \nabla\Phi^h, \det \nabla\Psi^h > 0$, where $\nabla \Phi^h$ and $\nabla \Psi^h$ denote the Fr\'echet derivatives of $\Phi^h$ and $\Psi^h$, respectively.

Finally, we assume that both $\GDh$ and $\GNh$ are unions of full faces
of elements $K \in \CT_h$. In other words, every mesh face either sits in $\GDh$
or in $\GNh$.

\begin{figure}[h]

\begin{tikzpicture}[scale=3]
\draw[ultra thick] (-1, 1) -- (1, 1);
\draw[ultra thick] (-1, 0) -- (1, 0);

\draw[ultra thick]        (-1, 0) to[in=180,out=180] (-1,1);
\draw[ultra thick,dashed] ( 0, 0) to[in=-110,out=70] ( 0,1);
\draw[ultra thick]        ( 1, 0) to[in=0,out=0]     ( 1,1);

\draw ( 0.0,0.0) node[anchor=north] {$\Omega$};
\draw (-0.5,0.5) node {$Q_-$};
\draw ( 0.5,0.5) node {$Q_+$};
\begin{scope}[xshift=2.5cm]
\draw[ultra thick] (-1, 1) -- (1, 1);
\draw[ultra thick] (-1, 0) -- (1, 0);

\draw[ultra thick]        (-1, 0) -- (-1,1);
\draw[ultra thick,dashed] ( 0, 0) -- ( 0,1);
\draw[ultra thick]        ( 1, 0) -- ( 1,1);

\draw (0,0) node[anchor=north] {$\Omega^h$};
\draw (-0.5,0.5) node {$Q_-^h$};
\draw ( 0.5,0.5) node {$Q_+^h$};
\end{scope}
\end{tikzpicture}

\begin{tikzpicture}[scale=3]
\draw[ultra thick] (-1, 1) -- (1, 1);
\draw[ultra thick] (-1, 0) -- (1, 0);

\draw[ultra thick]        (-1, 0) to[in=180,out=180] (-1,1);
\draw[ultra thick]        ( 1, 0) to[in=0,out=0]     ( 1,1);

\draw( 0, 0) to[in=-110,out=70] ( 0,1);


\draw (-1, 0) -- ( 0,1);
\draw ( 0, 0) -- ( 1,1);

\draw (-1, 1) -- ( 0,0);
\draw ( 0, 1) -- ( 1,0);

\draw ( 0.0,0.0) node[anchor=north] {$\{\widetilde K\}_{K \in \CT_h}$};

\draw (-0.75,0.50) node {$\widetilde K_1$};
\draw (-0.50,0.75) node {$\widetilde K_2$};
\draw (-0.25,0.50) node {$\widetilde K_3$};
\draw (-0.50,0.25) node {$\widetilde K_4$};

\draw ( 0.25,0.50) node {$\widetilde K_5$};
\draw ( 0.50,0.75) node {$\widetilde K_6$};
\draw ( 0.75,0.50) node {$\widetilde K_7$};
\draw ( 0.50,0.25) node {$\widetilde K_8$};

\draw[thick,->,rounded corners=1cm] (0.70,0.9) -- (1.75,1.1)-- (2.80,0.9);
\draw (1.75,1.1) node[anchor=south] {$\Phi^h|_{\tK_6}$};

\draw[thick,<-,rounded corners=1cm] (0.70,0.1) -- (1.75,-0.1)-- (2.80,0.1);
\draw (1.75,-0.1) node[anchor=north] {$\Psi^h|_{K_8}$};
\begin{scope}[xshift=2.5cm]
\draw[ultra thick] (-1, 1) -- (1, 1);
\draw[ultra thick] (-1, 0) -- (1, 0);

\draw[ultra thick]        (-1, 0) -- (-1,1);
\draw[ultra thick]        ( 1, 0) -- ( 1,1);

\draw (-1, 0) -- (-1,1);
\draw ( 0, 0) -- ( 0,1);
\draw ( 1, 0) -- ( 1,1);

\draw (-1, 0) -- ( 0,1);
\draw ( 0, 0) -- ( 1,1);

\draw (-1, 1) -- ( 0,0);
\draw ( 0, 1) -- ( 1,0);

\draw ( 0.0,0.0) node[anchor=north] {\color{white} $\{\widetilde K\}_{K \in \CT_h}$};
\draw ( 0.0,0.0) node[anchor=north] {$\CT_h$};

\draw (-0.75,0.50) node {$K_1$};
\draw (-0.50,0.75) node {$K_2$};
\draw (-0.25,0.50) node {$K_3$};
\draw (-0.50,0.25) node {$K_4$};

\draw ( 0.25,0.50) node {$K_5$};
\draw ( 0.50,0.75) node {$K_6$};
\draw ( 0.75,0.50) node {$K_7$};
\draw ( 0.50,0.25) node {$K_8$};
\end{scope}
\end{tikzpicture}

\caption{Example of a curved domain $\Omega$ with two subdomains $Q_\pm$. The
mesh $\CT_h$ consisting of straight elements induces an approximate domain $\Omega^h$ partitioned
into $Q_\pm^h$. Here, for each $K \in \CT_h$, $\mathcal F_K: \widehat K \to K$ is the affine map
mapping the reference element into $K$. The mapping $\Psi^h$ maps $\Omega^h$ onto $\Omega$,
and, in particular, maps each element $K_j$ onto $\widetilde K_j$. In this particular example,
$\Psi^h|_{K_j}$ is the identity map for $j=2,4,6$ and $8$, and is a non-affine mapping on
the remaining elements. $\Phi^h$ correspondingly maps $\widetilde K_j$ onto $K_j$.}
\label{figure_example}
\end{figure}
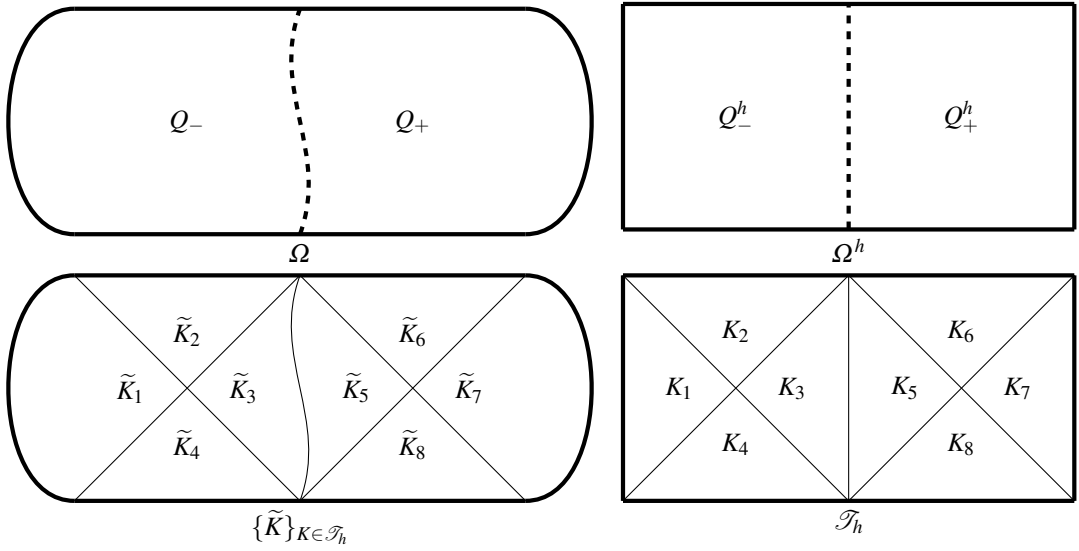

Fixing a polynomial degree $p \geq 1$, we associate with $\CT_h$ the following
finite element space:
\begin{equation*}
V_h
\eq
\left \{
v_h \in H^1_{\GDh}(\Omegah)
\; | \;
v_h \circ \MAP_K \in \CP_p(\hK)
\right \}.
\end{equation*}
Note that elements of $V_h$ need not be piecewise polynomials, since the
map $\MAP_K$ is not polynomial in general. Since piecewise polynomial spaces
enjoy excellent approximation properties, we therefore want $\MAP_K$ to be close to being affine, meaning that the mesh elements are not too distorted.

We can make the requirements needed for the maps $\MAP_K$ and $\Psi_K^h$ precise as follows (compare to \cite[Theorem 1 and Propositions 2 and 3]{lenoir_1986a}).

\begin{assumption}[Curved finite element space]
\label{assumption_curved_fem}
The elements $K \in \CT_h$ are not too distorted,
in the sense that the maps $\MAP_K$ satisfy
\begin{equation}
\label{eq_mapK}
\|\nabla^s\MAP_K\|_{L^\infty(\hK)}
\leq
C 
L\left (\frac{h_K}{L}\right )^s,
\quad\tand\quad
\|\nabla^s(\MAP_K^{-1})\|_{L^\infty(K)}
\leq
C h_K^{-s},
\end{equation}
for $1 \leq s \leq p+1$, where $h_K$ is the diameter of $K$.

The elements $K \in \CT_h$ approximate the geometry well, in the sense that 
\begin{equation}\label{eq_bound_Psi}
\|\nabla^s(\Psi_K^h-I)\|_{L^\infty(K)} \leq C 
\left(\frac{h_{K}}{L}\right)^q h_{K}^{1-s},
\qquad
\|\det (\nabla \Psi_K^h)-1\|_{L^\infty(K)} \leq C \left(\frac{h_{K}}{L}\right)^q,
\end{equation}
and
\begin{equation}\label{eq_bound_Phi}
\|\nabla^s(\Phi_K^h-I)\|_{L^\infty(\tK)}
\leq C 
\left(\frac{h_{K}}{L}\right)^q h_{K}^{1-s},
\qquad
\|\det (\nabla \Phi_K^h)-1\|_{L^\infty(\tK)} \leq C \left(\frac{h_{K}}{L}\right)^q,
\end{equation}
for $0 \leq s \leq q+1$.
\end{assumption}

The requirements in Assumption \ref{assumption_curved_fem} are easily met in practice.
Specifically, given a straight simplicial partition $\CT_h^\dagger$ that nodally conforms
to $\Part$ and $\Omega$, suitable maps $\MAP_K$ are constructed in \cite[\S3.2--3.3]{lenoir_1986a}.
Note that constructing $\CT_h^\dagger$ is in turn easily done with standard meshing
software like {\tt gmsh} (see \cite{geuzaine_remacle_2009a}).

Under Assumption \ref{assumption_curved_fem}, there exists an interpolation operator
$\CI_h: H^2(\Parth) \cap H^1_{\GDh}(\Omegah) \to V_h$ such that
\begin{equation}\label{eq_interpolation1}
h^{-1} \big\|(I-\CI_h)v\big\|_{L^2(\Omegah)} + \big|(I-\CI_h)v\big|_{H^1(\Omegah)}
\leq
C h^p \N{v}_{H^{p+1}(\Parth)}
\end{equation}
for all $v \in H^{p+1}(\Parth) \cap H^1_{\GDh}(\Omegah)$; see
\cite{ciarlet_raviart_1972a,lenoir_1986a}.
(Note that the piecewise Sobolev spaces on $\Parth$ are defined exactly as the ones on $\Part$, with norm given by the analogue of \eqref{eq_norm_Part}.)

We finally introduce a discrete sesquilinear form:
\begin{equation*}
b_h(v_h,w_h)
\eq
\sum_{Q^h \in \Parth}
\Big \{
-k^2(\mu_Qv_h,w_h)_{L^2(Q^h)}
+(A_Q\nabla v_h,\nabla w_h)_{L^2(Q^h)}
\Big \}
\quad
\tfa v_h,w_h \in V_h,
\end{equation*}
and a discrete right-hand side
\beq\label{eq_discrete_rhs2}
\langle \psi_h, w_h\rangle = \sum_{Q \in \Part}(\data_Q, w_h)_{L^2(Q^h)}.
\eeq
(Note that since each $Q^h$ is exactly covered by $\CT_h$,
the coefficients of the matrix associated with $b_h$ 
and the vector associated with $\psi_h$ can be computed efficiently.)
The discrete problem then consists of finding $u_h \in V_h$ such that
\begin{equation}
\label{eq_curved_fem}
b_h(u_h,w_h) = \langle \psi_h,w_h \rangle \quad\tfa w_h \in V_h.
\end{equation}

\begin{remark}[Quadrature error]\label{rem:quadrature}
In practice, if $\mu_Q$, $A_Q$, and $g_Q$, are not piecewise constant,
the integrals in the definitions of $b_h$ and $\psi_h$ are further approximated
by a quadrature rule. 
One way of doing this, which fits immediately into the framework above (and is also discussed in \cite[\S6]{lenoir_1986a}), is to introduce 
\beq\label{eq:sharp1}
b_h^\sharp(v_h,w_h)
\eq
\sum_{Q^h \in \Parth}
\Big \{
-k^2(\CJ_h(\mu_Q)v_h,w_h)_{L^2(Q^h)}
+(\CJ_h(A_Q)\nabla v_h,\nabla w_h)_{L^2(Q^h)}
\Big \}
\quad
\tfa v_h,w_h \in V_h,
\eeq
where $\CJ_h$ is an interpolation operator with $\CJ_h: W^{q,\infty}(K)\to \MAP_K(\CP_{q-1}(\hK))$ and the approximation property
\beq\label{eq:hot1}
\N{ I-\CJ_h}_{W^{q,\infty}(K) \to L^\infty(K)}\leq C (h/L)^q.
\eeq
Thus
\beqs
|b_h(v_h,w_h)-b^\sharp_h(v_h,w_h)| \leq (h/L)^q \Big(\N{\mu_Q}_{W^{q,\infty}} + \N{A_Q}_{W^{q,\infty}}\Big)\N{v_h}_{H^1_k(\Omega^h)}\N{w_h}_{H^1_k(\Omega^h)}.
\eeqs
Note that, by the definition of $\CJ_h$, when the integrals in \eqref{eq:sharp1} are maped back to reference element, the integrands are polynomials, and so can be integrated exactly.

In particular, for $q=1$, \eqref{eq:hot1} is satisfied with $\CJ_h$ taking the midpoint value, leading to 
\beq\label{eq:hot2}
b_h^\sharp(v_h,w_h)
\eq
\sum_{Q^h \in \Parth}\sum_{K\in Q^h}
\Big \{
-k^2(\mu_Q)(x_K)(v_h,w_h)_{L^2(K)}
+\big((A_Q)(x_K))\nabla v_h,\nabla w_h\big)_{L^2(K)}
\Big \},
\eeq
which is how $b(\cdot,\cdot)$ was approximated in the numerical experiments in \S\ref{sec:num}.

Finally, we note that the paper \cite{DrKeWo:20} proves bounds on $\Cquad$
\eqref{eq_definition_Cquad} for a particular method -- ``the surrogate matrix methodology" --
in the setting of isogeometric analysis; see \cite[Theorem 4.1]{DrKeWo:20}.
\end{remark}

\begin{remark}[Non-conforming methods]
The framework in \S\ref{sec:abstract} implements a so-called ``first'' Strang lemma, which is geared 
to dealing with a ``wrong'' sesquilinear form $b_h$ defined on the ``right'' variational space
(i.e., $V_h \subset H^1)$. On the other hand, non-conforming methods are dealt with a ``second''
Strang lemma (see, e.g., \cite[Theorem 31.1]{Ci:02}), in which $b$ is not approximated, but rather extended to work in a non-conforming
space (i.e. $V_h \not \subset H^1$). Using a second Strang lemma in conjunction with the 
framework in \S\ref{sec:abstract} is certainly possible, although it requires the introduction of extra quantities to control the
non-conformity. The analysis of a non-conforming method for Helmholtz problems through a second
Strang lemma is done, e.g., in~\cite{chaumont2024duality}.
\end{remark}

\subsection{A mapped finite element space}\label{sec:mapped}

The setting introduced above for curved finite elements does not immediately fit
our abstract framework. Indeed, $b_h$ contains exact integrals
on a wrong domain $\Omegah$, rather that inexact integrals on the right domain
$\Omega$. Following \cite{lenoir_1986a}, this is easily remedied by introducing
an abstract mapped finite element space. Notice that this space is never used in
actual computations; rather, it is a convenient tool for the analysis.

For $v_h \in V_h$, let 
\beq\label{eq_transformed_vh}
\tv_h \eq v_h \circ \Psi_h^{-1} \eq v_h \circ \Phi^h;
\eeq
recall from Section \ref{sec_4.2} that $\Psi^h: \overline{\Omegah} \to \overline{\Omega}$,
so that $\tv_h :\Omega\to \mathbb{C}$ and $\tv_h \in H^1_{\GD}(\Omega)$. 
The mapped finite element space $\tV_h \eq \{\tv_h; \; v_h \in V_h \}$ is therefore conforming, i.e., 
$\tV_h \subset H^1_{\GD}(\Omega)$.
Furthermore, instead of considering the solution $u_h$ of the variational problem
\eqref{eq_curved_fem}, we can (abstractly) solve a variational problem set in
$\Omega$ with the mapped finite element space $\tV_h$. Specifically, we can equivalently
consider the problem of finding
$\tu_h \in \tV_h$ such that
\begin{equation}\label{eq_lastday1}
\tb_h(\tu_h,\tw_h) = \langle \widetilde \psi_h,\tw_h \rangle
\qquad
\tfa\, \tw_h \in \tV_h,
\end{equation}
where
\begin{align}\nonumber
\tb_h(\tu_h,\tw_h)
&\eq
\sum_{Q \in \Part}
\Big\{
-k^2(\mu^h_Q \tu_h,\tw_h)_{L^2(Q)}
+(A^h_Q\nabla \tu_h,\nabla \tw_h)_{L^2(Q)}
\Big \}\\
&=
-k^2(\mu_h\tu_h,\tw_h)_{L^2(\Omega)} +(A^h\nabla \tu_h,\nabla \tw_h)_{L^2(\Omega)}
\label{eq_tildeb}
\end{align}
with 
\beq\label{eq_AQh}
\mu_Q^h \eq |\det \nabla \Phi^h| \big(\mu_Q \circ \Phi^h\big),\qquad 
A_Q^h \eq |\det \nabla \Phi^h| \nabla \Psi^h (A_Q\circ\Phi^h) (\nabla \Psi^h)^T
\eeq
and $\mu_h|_Q \eq \mu^h_Q$ and $A^h|_Q \eq A^h_Q$ for all $Q \in \Part$.
Similarly, 
\beq\label{eq_discrete_rhs}
\langle \widetilde{\psi}_h, \tw_h \rangle := \sum_{Q \in \Part} \big( \data^h_Q, \tw_h\big)_{L^2(\Omega)}
\quad \text{ with } \quad 
\data_Q^h \eq  |\det \nabla \Phi^h| \big(\data_Q \circ \Phi^h\big).
\eeq
These definitions are made so that, by the chain rule (see, e.g., \cite[Equation 9.8a]{ErGu:21}),
\beqs
b_h(v_h,w_h) = \tb_h(\tv_h,\tw_h)\quad \tfa v_h,w_h \in V_h,
\eeqs
and thus $u_h = \tu_h \circ \Psi^h$ by \eqref{eq_curved_fem} and \eqref{eq_lastday1}.

It is clear from Assumption \ref{assumption_curved_fem} and \eqref{eq_interpolation1} that
if we set $\tCI_h = \Psi^h \circ \CI_h \circ (\Psi^h)^{-1}$, we obtain a
suitable interpolation operator onto $\tV_h$. Specifically,
\cite[Lemma 7]{lenoir_1986a} shows that\footnote{Strictly
speaking, \cite[Lemma 7]{lenoir_1986a} establishes only the bound on the $H^1$ semi-norm,
but the proof for the bound on the $L^2$ norm is analogous.}
\begin{equation}
\label{eq_mapped_interpolant}
h^{-1} \|v-\tCI_h v\|_{L^2(\Omega)} + |v-\tCI_h v|_{H^1(\Omega)}
\leq
C h^p \N{v}_{H^{p+1}(\Part)} \quad\tfa v \in H^{p+1}(\Part) \cap H^1_{\GD}(\Omega).
\end{equation}

\subsection{Bounds on $\Cquad$ \eqref{eq_definition_Cquad} and $\Cdrhs$ \eqref{eq_definition_Cdrhs}}

With our framework established, we are ready to provide bounds
on 
$\Cquad$ \eqref{eq_definition_Cquad} and $\Cdrhs$ \eqref{eq_definition_Cdrhs}. The proof for $\Cquad$ closely follows
\cite[Lemma 8]{lenoir_1986a} whereas the auxiliary result established
in Appendix \ref{section_change_variable} is used for $\Cdrhs$.

\begin{theorem}[Bounds on $\Cquad$ \eqref{eq_definition_Cquad} and $\Cdrhs$ \eqref{eq_definition_Cdrhs}]\label{thm_gammaq}
If $V_h$ satisfies Assumption \ref{assumption_curved_fem}, then
\begin{equation}
\Cquad
\leq
C \left (\frac{h}{L}\right )^q W(\mu,A),
\qquad
\Cdrhs \|\psi\|_{(H^1_k)^\star}
\leq
\frac{C}{M}
\left (\frac{h}{L}\right )^q
\frac{1}{k}
\bigg (
\sum_{Q \in \Part}
\|\data_Q\|_{L^2(\mathbb R^d)}^2
+
h^2\|\nabla \data_Q\|_{L^2(\mathbb R^d)}^2
\bigg )^{1/2}
\label{eq_gammaq_bounds}
\end{equation}
with
\begin{equation*}
W(\mu,A)
:=
\frac{1}{M}
\max_{Q \in \Part}
\Big(
\|\mu_Q\|_{L^\infty(\mathbb R^d)} + h\|\nabla \mu_Q\|_{L^\infty(\mathbb R^d)}
+
\|A_Q\|_{L^\infty(\mathbb R^d)} + h\|\nabla A_Q\|_{L^\infty(\mathbb R^d)}
\Big).
\end{equation*}
\end{theorem}

\begin{proof}
We first prove the bound on $\Cquad$, and observe that we only need to show that
\begin{equation}
\label{eq_approximation_coef}
\|\mu-\mu^h\|_{L^\infty(\Omega)}
+
\|A-A^h\|_{L^\infty(\Omega)}
\leq
C \left (
\frac{h}{L}
\right )^qW(\mu,A),
\end{equation}
since then, by the definition of $\tb_h$ in \eqref{eq_tildeb}, 
\begin{equation}
\label{eq_geometric_error}
|b(\tv_h,\tw_h)-\tb_h(\tv_h,\tw_h)|
\leq
C \left (
\frac{h}{L}
\right )^q W(\mu,A)
\|\tv_h\|_{H^1_k(\Omega)}
\|\tw_h\|_{H^1_k(\Omega)}
\end{equation}
for all $\tv_h,\tw_h \in \tV_h$.
The bound on $\Cquad$ then follows immediately from its definition \eqref{eq_definition_Cquad} and \eqref{eq_geometric_error}.

We now show how \eqref{eq_approximation_coef} follows from Assumption \ref{assumption_curved_fem}.
For the coefficient $A$, if $Q \in \Part$ and $K \in \CT_h$ with $\tK \subset \overline{Q}$,
then, by \eqref{eq_AQh} and the assumption that $\det \nabla \Phi^h>0$,
\begin{align*}
A_Q-A_Q^h
&=
A_Q-(\det \nabla \Phi^h) \nabla \Psi^h (A_Q \circ \Phi^h) (\nabla \Psi^h)^T,
\\
&=
(1-\det\nabla\Phi^h)A_Q-\det \nabla \Phi^h \left \{
A_Q-\nabla \Psi^h (A_Q \circ \Phi^h) (\nabla \Psi^h)^T\right \}
\end{align*}
so that
\begin{align}\nonumber
\|A_Q-A_Q^h\|_{L^\infty(\tK)}
&\leq
\|1-\det\nabla\Phi^h\|_{L^\infty(\tK)} \|A_Q\|_{L^\infty(\Rea^d)}
\\
&\hspace{3cm}
+
\|\det \nabla \Phi^h\|_{L^\infty(\tK)}
\|A_Q-\nabla \Psi^h (A_Q \circ \Phi^h) (\nabla \Psi^h)^T\|_{L^\infty(\tK)}
\nonumber
\\
&\leq
C \left (
\left (\frac{h}{L}\right )^q\|A_Q\|_{L^\infty(\tK)}
+
\|A_Q-\nabla \Psi^h (A_Q \circ \Phi^h) (\nabla \Psi^h)^T\|_{L^\infty(\tK)}
\right )
\label{eq_T1}
\end{align}
by the second bound in \eqref{eq_bound_Phi}.
We now work on the second term on the right-hand side of \eqref{eq_T1}.
Specifically, with $\MI$ the $d\times d$ identity matrix,
\begin{align*}
A_Q-\nabla \Psi^h (A_Q \circ \Phi^h) (\nabla \Psi^h)^T
&=
(\MI-\nabla \Psi^h) A_Q
+
\nabla \Psi^h \big[A_Q-(A_Q \circ \Phi^h) (\nabla \Psi^h)^T\big]
\\
&=
(\MI-\nabla \Psi^h) A_Q
+
\nabla \Psi^h \big[A_Q(\MI-(\nabla \Psi_h)^T)\big]
\\
&\hspace{4.5cm}
+
\nabla \Psi^h \big(A_Q-A_Q \circ \Phi^h) (\nabla \Psi^h)^T,
\end{align*}
and therefore
\begin{align}
\|A_Q-\nabla \Psi^h (A_Q \circ \Phi^h) (\nabla \Psi^h)^T\|_{L^\infty(\tK)}
&\leq
\|\MI-\nabla \Psi^h\|_{L^\infty(\tK)} \|A_Q\|_{L^\infty(\Rea^d)}
\nonumber
\\
&+
C\|\nabla \Psi^h\|_{L^\infty(\tK)}\|A_Q\|_{L^\infty(\Rea^d)}
\|\MI-(\nabla \Psi^h)^T\|_{L^\infty(\tK)}
\nonumber
\\
&+
C\|\nabla \Psi^h\|_{L^\infty(\tK)}^2\|A_Q-A_Q \circ \Phi^h\|_{L^\infty(\tK)}
\nonumber
\\
&\leq
C \left (\left (\frac{h}{L}\right )^q \|A_Q\|_{L^\infty(\Rea^d)} +  \|A_Q-A_Q \circ \Phi^h\|_{L^\infty(\tK)}\right )
\label{eq_T2}
\end{align}
by the first bound in \eqref{eq_bound_Psi} with $s=1$. We estimate the remaining term $ \|A_Q-A_Q \circ \Phi^h\|_{L^\infty(\tK)}$ by 
using the mean-value inequality 
and the first bound in \eqref{eq_bound_Phi} with $s=0$ to find that 
\begin{equation}
\label{eq_T3}
\|A_Q-A_Q \circ \Phi^h\|_{L^\infty(\tK)}
\leq
C\|I-\Phi^h\|_{L^\infty(\tK)} \|\nabla A_Q\|_{L^\infty(\Rea^d)}
\leq
C \left (\frac{h}{L}\right )^q h\|\nabla A_Q\|_{L^\infty(\Rea^d)}.
\end{equation}

Combining \eqref{eq_T1}, \eqref{eq_T2}, \eqref{eq_T3}, and observing that 
these bounds are valid for all $K \in \CT_h$, we obtain that 
\begin{equation}
\label{eq_approx_A}
\|A-A^h\|_{L^\infty(\Omega)}
\leq
C \left (
\frac{h}{L}
\right )^q
\max_{Q \in \Part}
\Big(
\|A_Q\|_{L^\infty(\mathbb R^d)} + h\|\nabla A_Q\|_{L^\infty(\mathbb R^d)}
\Big)
\end{equation}
The corresponding result for $\mu$, namely,
\begin{equation}
\label{eq_approx_mu}
\|\mu-\mu^h\|_{L^\infty(\Omega)}
\leq
C \left(
\frac{h}{L}
\right )^q
\max_{Q \in \Part}
\Big (
\|\mu_Q\|_{L^\infty(\mathbb R^d)} + h\|\nabla \mu_Q\|_{L^\infty(\mathbb R^d)}
\Big )
\end{equation}
follows in an almost-identical way; we have therefore proved \eqref{eq_approximation_coef}, and the proof of the bound on $\Cquad$ is complete.

To prove the bound on $\Cdrhs$ \eqref{eq_definition_Cdrhs}, we first observe that, 
by \eqref{eq_discrete_rhs},
\begin{align*}
\langle \widetilde \psi_h, \tv_h \rangle
&=
\sum_{Q \in \Part}
\int_Q  \det(\nabla \Phi^h) (\data_Q \circ \Phi^h) \tv_h
\\
&
=
\sum_{Q \in \Part}
\int_Q \big(\det(\nabla \Phi^h)-1\big) (\data_Q \circ \Phi^h) \tv_h
+
\sum_{Q \in \Part}
\int_Q \big(
\data_Q \circ \Phi^h -\data_Q
\big) \tv_h
+
\sum_{Q \in \Part}
\int_Q \data_Q \tv_h,
\end{align*}
so that, by \eqref{eq_discrete_rhs2}, 
\begin{equation}\label{eq_car1}
\langle \widetilde \psi_h-\psi_h,\tv_h \rangle
=
\sum_{Q \in \Part}
\int_Q \big(\det(\nabla \Phi^h)-1\big) (\data_Q \circ \Phi^h) \tv_h
+
\sum_{Q \in \Part}
\int_Q \big(
\data_Q \circ \Phi^h -\data_Q
\big)\tv_h.
\end{equation}
For a fixed $Q \in \Part$,
\begin{align}\nonumber
\left|\int_Q (\det \nabla \Phi^h-1) (\data_Q \circ \Phi^h)\tv_h\right|
&=
\left |
\sum_{\substack{K \in \CT_h \\ \tK \subset Q}}
\int_{\tK} \big(\det \nabla \Phi^h_K-1\big) \big(\data_Q \circ \Phi^h_K\big) \tv_h
\right |
\\ \nonumber
&\leq
\sum_{\substack{K \in \CT_h \\ \tK \subset Q}}
\|\det(\nabla \Phi^h_K)-1\|_{L^\infty(\tK)}
\|\data_Q \circ \Phi^h_K\|_{L^2(\tK)} \|\tv_h\|_{L^2(\tK)}
\\
&\leq
C\left (\frac{h}{L}\right )^q
\|\data_Q\|_{L^2(\mathbb R^d)} \|\tv_h\|_{L^2(Q)}\label{eq_car2}
\end{align}
by the second bound in \eqref{eq_bound_Phi}. 
Similarly, by
Lemma \ref{lemma_change_variable} and the first bound in  \eqref{eq_bound_Phi} with $s=0$ and $s=1$, 
\begin{align}\label{eq_car3}
\left|\int_Q \big(\data_Q \circ (\Phi^h-I)\big) \tv_h\right|
&\leq
C\left (\frac{h}{L}\right )^{q}
h
\|\nabla \data_Q\|_{L^2(\mathbb R^d)}\|\tv_h\|_{L^2(Q)}.
\end{align}
Using \eqref{eq_car2} and \eqref{eq_car3} in \eqref{eq_car1}, we obtain
\begin{align*}
|\langle \widetilde \psi_h-\psi_h,\tv_h \rangle|
&\leq
C \left (\frac{h}{L}\right )^q
\bigg(
\sum_{Q \in \Part}
\|\data_Q\|_{L^2(\mathbb R^d)}^2 + h^2\|\nabla \data_Q\|_{L^2(\mathbb R^d)}^2
\bigg)^{1/2}
\|\tv_h\|_{L^2(\Omega)}\\
&\leq
\frac{C}{k} \left (\frac{h}{L}\right )^q
\bigg(
\sum_{Q \in \Part}
\|\data_Q\|_{L^2(\mathbb R^d)}^2 + h^2\|\nabla \data_Q\|_{L^2(\mathbb R^d)}^2
\bigg )^{1/2}
\|\tv_h\|_{H^1_k(\Omega)}.
\end{align*}
The second bound in \eqref{eq_gammaq_bounds} then follows from the definition of $\Cdrhs$ \eqref{eq_definition_Cdrhs}.
\end{proof}

%
%

\subsection{Bounds on $\Caprx$ \eqref{eq_definition_Caprx} and $\Caprxop$ \eqref{eq_definition_Caprxop}}\label{sec:aprxop}

The next result uses the following $k$-weighted norms. 
For a Lipschitz domain $D$, let 
\beq\label{eq_weighted}
\N{u}^2_{H^m_k(D)}:= \sum_{|\alpha|\leq m} k^{2(1-|\alpha|)} \N{\partial^\alpha u}^2_{L^2(D)}
\quad \tfor m\in\mathbb{Z}^+
\eeq
(observe that this definition is equivalent to \eqref{eq_norm} when $m=1$) 
and let 
\beq\label{eq_weighted2}
\N{u}^2_{H^m_k(\Part)} := \sum_{Q\in \Part} \N{u}^2_{H^m_k(Q)}\quad \tfor m\in\mathbb{Z}^+.
\eeq

\begin{lemma}[Existence of a smoothing $\operator$ satisfying the G\aa rding-type inequality \eqref{eq_Garding}]\label{lem_GS}
Suppose that 
$A$ satisfies Assumption \ref{ass_strong_ellipticity} and
$A$, $\mu$, and $\partial\Omega$ satisfy the regularity requirements in Assumption \ref{ass_regularity},
for some $\ell \in \mathbb{Z}^+$. Let $b$ be the sesquilinear form \eqref{eq_actual_b}.

Then there exists a self-adjoint $\operator :L^2(\Omega)\to L^2(\Omega)$ such that given $k_0>0$ there exists $C>0$ such that
the G\aa rding-type inequality \eqref{eq_Garding} is satisfied and 
$\operator : L^2(\Omega)\to H^{\ell+1}(\Part)\cap H^1_{\GD}(\Omega)$ 
with 
\beq\label{eq_operator_regularity}
\N{\operator}_{L^2(\Omega)\to L^2(\Omega)}\leq C 
\quad\tand\quad
\N{\operator }_{L^2(\Omega)\to H^n_k(\Part)} \leq Ck \quad\tfa k\geq k_0, \quad 1\leq n\leq \ell+1.
\eeq
Furthermore $\tsol$ defined by \eqref{eq_tsol} satisfies $\tsol : H^{\ell-1}(\Part)\cap H^1_{\GD}(\Omega)
\to H^{\ell+1}(\Part)$ with 
\beq\label{eq_tsol_regularity}
\big\|\tsol\big\|_{H^{\ell-1}_k(\Part)\to H^{\ell+1}_k(\Part)} \leq C k^{-2} \quad \tif \ell \geq 2,
\eeq
and 
\beq\label{eq_tsol_regularity2}
\big\|\tsol\big\|_{L^2(\Omega)\to H^{2}_k(\Part)} \leq C k^{-1} \quad \tif \ell =1.
\eeq
\end{lemma}

We now 
proceed with the operator $\operator$ introduced in Lemma \ref{lem_GS}.

\begin{theorem}[Bounds on $\Caprx$ \eqref{eq_definition_Caprx} and  $\Caprxop$ \eqref{eq_definition_Caprxop}]\label{thm_main_event1}
Suppose that 
$A$ satisfies Assumption \ref{ass_strong_ellipticity}, that $A$, $\mu$, and $\partial\Omega$ satisfy the regularity requirements in Assumption \ref{ass_regularity} for some $\ell \in \mathbb{Z}^+$, and 
that $V_h$ satisfies Assumption \ref{assumption_curved_fem} for some $p\leq \ell$. Suppose that $kh\leq C_0$ for some $C_0>0$. 
Then given $k_0>0$ there exists $C>0$ such that, for $k\geq k_0$, with $\Caprx$ defined by \eqref{eq_definition_Caprx} with the finite-element space equal to $\widetilde{V}_h$ and the operator $\operator$ given by Lemma \ref{lem_GS},
\beq\label{eq_Cstab_bound}
\Cstab\leq C k^\alpha,
\eeq
\beq\label{eq_Caprx_bound}
\Caprx \leq C k^\alpha(kh)^p,
\eeq
and
%
%
%
\beq\label{eq_definition_Caprxop_bound}
 \Caprxop
 \leq C (kh)^p.
\eeq
\end{theorem}

\bpf[Proof of the bound \eqref{eq_Cstab_bound} using Lemma \ref{lem_GS}]
Since $\Cstab := k \| \asol \operator\|_{L^2(\Omega)\to H^1_k(\Omega)}$ \eqref{eq_definition_Cstab}, 
if we can show that 
\beq\label{eq_asol_bound}
\| \asol\|_{L^2(\Omega)\to H^1_k(\Omega)}\leq C k^{\alpha-1},
\eeq
 then 
the bound \eqref{eq_Cstab_bound} follows from the combination of \eqref{eq_asol_bound} and the first bound in 
\eqref{eq_operator_regularity}. Since the $L^2\to L^2$ norm of the adjoint solution operator equals the $L^2\to L^2$ norm of the solution operator (by properties of the adjoint), the solution-operator bound \eqref{eq:alpha} implies that $\| \asol\|_{L^2(\Omega)\to L^2(\Omega)}\leq C k^{\alpha-2}$; the bound \eqref{eq_asol_bound} then follows from the G\aa rding inequality \eqref{eq_Garding_standard} and the first bound in \eqref{eq_operator_regularity}.
\epf

\

\bpf[Proof of the bound \eqref{eq_Caprx_bound} using Lemma \ref{lem_GS}]
When $p\geq 2$, we claim that it is sufficient to prove that
\beq\label{eq_STP1}
\big\| \asol\big\|_{H^{p-1}_k(\Part) \to H^{p+1}_k(\Part)}\leq C k^{\alpha-2}.
\eeq
(the temporary exclusion of the case $p=1$ is because the $H^m_k$ norm defined in \eqref{eq_weighted} equals $k$ times the $L^2$ norm when $m=0$, and so we need to treat separately the case when $\asol$ acts on $L^2$).
Indeed, by the definition of $\Caprx$ \eqref{eq_definition_Caprx}, the interpolation bound \eqref{eq_mapped_interpolant},
 the second bound in \eqref{eq_operator_regularity} with $n=p-1$, 
  the definition of the weighted norm \eqref{eq_weighted2}, and \eqref{eq_STP1}, 
\begin{align*}
\Caprx:= k\big\| (I-\pi_h) \asol \operator\big\|_{L^2(\Omega)\to H^1_k(\Omega)} &
\leq k C (1+kh) h^{p} \big\| \asol\big\|_{H^{p-1}_k(\Part)\to H^{p+1}(\Part)}k
\\
&\leq k C (1+kh) (kh)^p k^{\alpha-2} 
 k \leq C (kh)^p k^\alpha. 
\end{align*}
We therefore now prove \eqref{eq_STP1}. Given $\phi\in L^2(\Omega)$, by the definition of $\asol$ \eqref{eq_tsol}, $v:= \asol\phi \in H^1_{\GD}(\Omega)$ is the solution of the equation 
\beqs
- k^2 \overline{\mu} v-\nabla\cdot\big(\overline{A}^T\nabla v\big) =\phi.
\eeqs
Since $A,\mu$, and $\partial\Omega$ satisfy Assumption \ref{ass_regularity}, 
by elliptic regularity (see, e.g., \cite[Theorem 4.20]{Mc:00}), the
bound on $\asol$ \eqref{eq_asol_bound},
and the definition of $\|\cdot\|_{H^1_k(\Omega)}$ \eqref{eq_norm},
for $1\leq m\leq \ell-1$,
\begin{align*}
|v|_{H^{m+2}(\Part)}\leq C \Big( \big\|\nabla \cdot \big( \overline{ A}^T\nabla v \big)\big\|_{H^m(\Part)}+
\N{v}_{H^1(\Omega)}
\Big)
&\leq C\Big( \big\| k^2\overline{\mu} v + \phi\big\|_{H^m(\Part)}+
\N{v}_{H^1(\Omega)}
\Big)\\
&\leq C\Big( k^2 \N{v}_{H^m(\Part)}+ \N{\phi}_{H^m(\Part)}+k^{\alpha-1} \N{\phi}_{L^2(\Omega)}\Big).
\end{align*}
After multiplying by $k^{-m-1}$ and using the definitions of the weighted norms \eqref{eq_weighted} and \eqref{eq_weighted2}, we obtain that 
\beq\label{eq_ER2a}
k^{-m-1}|v|_{H^{m+2}(\Part)} \leq C \Big( \N{v}_{H^m_k(\Part)} + k^{-2} \N{\phi}_{H^m_k(\Part)}+k^{\alpha-2} \N{\phi}_{L^2(\Omega)}\Big)\quad\tfor
1\leq m\leq \ell-1.
\eeq
By iteratively applying \eqref{eq_ER2a}, 
we obtain that $\N{v}_{H^{n+2}_k(\Part)}\leq C k^{\alpha-2} \N{\phi}_{H^{n}_k(\Part)}$ for all $1\leq n\leq \ell-1$; 
since $p\leq \ell$, this last bound with $n=p-1$ implies the desired bound \eqref{eq_STP1} when $p\geq 2$.

The proof for the case $p=1$ is very similar, with \eqref{eq_STP1} replaced by 
\beq\label{eq_STP2}
\big\| \asol\big\|_{L^2(\Omega) \to H^{2}_k(\Part)}\leq C k^{\alpha-1},
\eeq
and the additional power of $k$ ``won back" by using the first bound in \eqref{eq_operator_regularity} instead of the second.
\epf

\

\bpf[Proof of the bound \eqref{eq_definition_Caprxop_bound} using Lemma \ref{lem_GS}]
By \eqref{eq_tpi_pi}, \eqref{eq_mapped_interpolant}, and the definition of the weighted norm \eqref{eq_weighted2},
\begin{align*}
\big\|(I-\tpi_h)v\big\|_{H^1_k(\Omega)}
&\leq C (1+kh) h^{p} \N{v}_{H^{p+1}(\Part)}\\
&\leq 
C (1+kh) (kh)^{p} \N{v}_{H^{p+1}_k(\Part)}
\quad\tfa v \in H^{p+1}(\Part)\cap H^1_{\GD}(\Omega).
\end{align*}
Combining this with \eqref{eq_tsol_regularity} and the second bound in
\eqref{eq_operator_regularity}, and recalling that $p\leq \ell$, we obtain that,
when $\ell\geq 2$,
\begin{align*}
\Caprxop& := k 
\big\|
(I-\tpi_h) \tsol\operator\big\|_{L^2(\Omega)\to H^1_k(\Omega)}\\
&\leq 
k\big\|I-\tpi_h\big\|_{H^{\ell+1}(\Part)\to H^1_k(\Omega)} \big\| \tsol\big\|_{H^{\ell-1}_k(\Part)\to H^{\ell+1}_k(\Part)}
 \big\|\operator\big\|_{L^2(\Omega)\to H^{\ell-1}_k(\Part)}\\
& \leq C k(1+kh)(kh)^\ell k^{-2}k;
\end{align*}
the result \eqref{eq_definition_Caprxop_bound} when $\ell\geq 2$ then follows.
The result when $\ell=1$ follows in a similar way using \eqref{eq_tsol_regularity2}
instead of \eqref{eq_tsol_regularity}, and with the additional power of $k$ again
``won back'' by using the first bound in \eqref{eq_operator_regularity} instead of the second.
\epf

\

\bpf[Proof of Lemma \ref{lem_GS}]
This result was first proved in \cite[Lemma 2.1]{GaSp:23}. For completeness, we give a (slightly different) proof here.  
Observe that Assumption \ref{ass_strong_ellipticity} implies that
$\Re A$ is a real, symmetric, positive definite matrix-valued function.
Combining this with the uniform positivity of $\Re \mu$ and
the spectral theorem (see, e.g., \cite[Theorem 2.37]{Mc:00}), we see that
there exist $\{\lambda_j^2\}_{j=1}^\infty$ and $\{ u_j\}_{j=1}^\infty$ with 
$\lambda^2_1 \leq \lambda^2_2 \leq \ldots$ with $0\leq \lambda_j^2 \to \infty$ as $j\to \infty$ 
and $u_j \in H^1_{\GD}(\Omega)$ with 
\beqs
( u_j, u_\ell )_{\Re \mu} := \big( (\Re \mu) u_j ,u_\ell \big)_{L^2(\Omega)} = \delta_{j \ell}
\eeqs
such that 
\beq\label{eq_eigenvalue1}
\big( (\Re A) \nabla u_j ,\nabla w\big)_{L^2(\Omega)} = \lambda^2_j \big( (\Re \mu) u_j, w\big)_{L^2(\Omega)} \quad \tfa w\in H^1_{\GD}(\Omega)
\eeq
i.e., 
\beq\label{eq_eigenvalue2}
- \nabla\cdot\big( (\Re A) \nabla u_j \big) = \lambda_j^2 (\Re \mu)u_j,
\eeq
and
\beq\label{eq_expansion}
v = \sum_{j=1}^\infty \big( v, u_j\big)_{\Re \mu} u_j \quad \tfa v \in L^2(\Omega).
\eeq
We record immediately the consequences of  \eqref{eq_expansion} and \eqref{eq_eigenvalue1} that 
\beq\label{eq_Parseval}
\big( (\Re\mu) v,v\big)_{L^2(\Omega)} = \sum_{j=1}^\infty \big|\big( v, u_j\big)_{\Re \mu} \big|^2
\quad\tand\quad
\big( (\Re A) \nabla v,\nabla v\big)_{L^2(\Omega)} = \sum_{j=1}^\infty \lambda_j^2 \big|\big( v, u_j\big)_{\Re \mu} \big|^2,
\eeq
respectively. 
Let 
\beq\label{eq_operator_def}
\widetilde\operator v := \sum_{\substack{j=1 \\\lambda_j^2\leq 2k^2}}^\infty \big( v, u_j\big)_{\Re \mu} u_j \quad \tfa v \in L^2(\Omega).
\eeq
and let 
\beq\label{eq_convert}
\operator := \Big(\sup_{x \in \Omega} \big(\Re \mu(x) \big)^{1/2} \Big)\, \widetilde\operator.
\eeq
Since $\widetilde\operator^2=\widetilde\operator$, the definition \eqref{eq_operator_def} implies that $\widetilde\operator$ has norm one in the $L^2$ norm weighted with $\Re \mu$, and thus the $L^2(\Omega)\to L^2(\Omega)$ bound in \eqref{eq_operator_regularity} holds.

We next prove that the $L^2(\Omega)\to H^n_k(\Omega)$ bound in \eqref{eq_operator_regularity} holds with $1\leq n\leq \ell+1$. 
Since $A,\mu$, and $\partial\Omega$ satisfy Assumption \ref{ass_regularity}, 
by elliptic regularity (see, e.g., \cite[Theorem 4.20]{Mc:00}), 
\begin{align*}
\N{u_j}_{H^2(\Part)} &
\leq C \Big( \N{ \nabla\cdot \big( (\Re A) \nabla u_j\big) }_{L^2(\Omega)} + \N{u}_{H^1(\Omega)}\Big)\leq C \Big( \lambda_j^2 \N{u_j}_{L^2(\Omega)}  + \N{u}_{H^1(\Omega)}\Big).
\end{align*}
By \eqref{eq_eigenvalue1} with $w=u_j$, $\|u_j\|_{H^1(\Omega)}\leq C|\lambda_j| \N{u_j}_{L^2(\Omega)}$, and thus 
\beq\label{eq_H2}
\N{u_j}_{H^2(\Part)} \leq C \lambda_j^{2}\N{u_j}_{L^2(\Omega)}.
\eeq
The bound \eqref{eq_H2} and interpolation (see, e.g. \cite[Theorems B.2 and B.8]{Mc:00})
imply that 
\beq\label{eq_H1}
\N{u_j}_{H^1(\Part)} \leq C |\lambda_j|\N{u_j}_{L^2(\Omega)}.
\eeq
By elliptic regularity, \eqref{eq_eigenvalue2}, and \eqref{eq_H1}, for $0\leq m\leq \ell+1$, 
\begin{align}\nonumber
\N{u_j}_{H^{m+2}(\Part)} &\leq C\Big(  \N{ \nabla\cdot \big( (\Re A) \nabla u_j\big) }_{H^m(\Part)} + \N{u_j}_{H^1(\Omega)}\Big)\\
&\leq C \Big( \lambda_j^2 \N{u_j}_{H^m(\Part)} + \N{u_j}_{H^1(\Omega)}\Big)\leq C \lambda_j^2\N{u_j}_{H^m(\Part)}\label{eq_ER1}
\end{align}
(where the upper limit of $m=\ell+1$ is determined by the regularity of $A,\mu,$ and $\partial \Omega$ in Assumption \ref{ass_regularity}).
The combination of \eqref{eq_H2}, \eqref{eq_H1}, and \eqref{eq_ER1} imply that 
\beq\label{eq_email1}
\N{u_j}_{H^{n}(\Part)} \leq C |\lambda_j|^{n}\N{u_j}_{L^2(\Omega)} \quad\tfa 0\leq n\leq \ell+1.
\eeq
Therefore, by the definition of $\widetilde\operator$ \eqref{eq_operator_def}, \eqref{eq_email1}, and \eqref{eq_Parseval},
\beqs
\big\|\widetilde\operator v\big\|^2_{H^{n}(\Part)} \leq C \sum_{\substack{j=1\\ \lambda_j^2 \leq 2k^2}} \lambda_j^{2n} \big| (v,u_j)_{\Re \mu}\big|^2 
\leq C k^{2n} \N{v}_{L^2(\Omega)}^2 \quad\tfa 0\leq n\leq \ell+1;
\eeqs
the bound \eqref{eq_operator_regularity} then follows $0\leq n\leq \ell+1$ by the definition  of the weighted norm \eqref{eq_weighted2}.

To prove the G\aa rding-type inequality \eqref{eq_Garding}, observe that, by \eqref{eq_Parseval} and the definition of $\widetilde\operator$ \eqref{eq_operator_def},
\begin{align}\nonumber
\Re \big(b(v,v)\big) + 2k^2 \big( \widetilde\operator v, \widetilde\operator v\big)_{\Re \mu}
&= (\Re b)(v,v) + 2k^2 \big( \widetilde\operator v, v\big)_{\Re \mu},\\ \nonumber
&\hspace{-1cm}= \sum_{j=1}^\infty\big( \lambda_j^2 -k^2\big) \big| (v,u_j)_{\Re \mu}\big|^2 
+2k^2\sum_{\substack{j=1\\ \lambda_j^2 \leq 2k^2}}^\infty  \big| (v,u_j)_{\Re \mu}\big|^2,\\
&\hspace{-1cm}= \sum_{\lambda_j^2 \leq 2k^2}\big( \lambda_j^2 +k^2\big) \big| (v,u_j)_{\Re \mu}\big|^2 
+\sum_{\lambda_j^2 > 2k^2} \big(\lambda_j^2 -k^2) \big| (v,u_j)_{\Re \mu}\big|^2.\label{eq_the_end1}
\end{align}
Now, 
\beq\label{eq_the_end2}
\text{ if } \quad\lambda_j^2 \geq 2k^2 \quad\text{ then }\quad (\lambda_j^2 -k^2) \geq \frac{1}{2}\lambda_j^2 \geq \frac{1}{4} (\lambda_j^2 +k^2).
\eeq
The definition \eqref{eq_operator_def} implies that
\beq\label{eq_last_one}
\big\|\operator v\big\|_{L^2(\Omega)} \geq \big\|\widetilde\operator v \big\|_{\Re \mu} \quad\tfa v \in L^2(\Omega).
\eeq
Therefore, combining \eqref{eq_last_one}, \eqref{eq_the_end1}, and \eqref{eq_the_end2}, we obtain that, for all $v\in H^1_{\GD}(\Omega)$,
\begin{align*}
\Re \big(b(v,v)\big) + 2k^2\N{ \operator v}^2_{L^2(\Omega)}
&\geq
\Re \big(b(v,v)\big) + 2k^2\big\|\widetilde{\operator} v\big\|^2_{\Re \mu}\\
&\geq 
\frac{1}{4}
 \sum_{j=1}^\infty\big( \lambda_j^2 +k^2\big) \big| (v,u_j)_{\Re \mu}\big|^2 \\
& \geq \frac{1}{4} \Big( \big( (\Re A) \nabla v,\nabla v\big)_{L^2(\Omega)}+ k^2 \big( (\Re\mu)v,v\big)_{L^2(\Omega)}\Big)
\geq \Ccoer \N{v}^2_{H^1_k(\Omega)},
\end{align*}
which is \eqref{eq_Garding} with
\beqs
\Ccoer := \frac{1}{4} \min \Big\{ \inf_{x\in\Omega}\big((\Re A)_*(x)\big)
,  \inf_{x\in\Omega}(\Re \mu(x)) \Big\}.
\eeqs
where, for $x\in \Omega$, $(\Re A)_*(x):= \inf_{\xi \in \mathbb{R}^d} (\Re A)(x)_{ij}\xi_i\xi_j$.

To complete the proof, we need to prove the bounds \eqref{eq_tsol_regularity}
and \eqref{eq_tsol_regularity2}; these are proved using elliptic regularity in
a very similar way to the proof of the bound \eqref{eq_Caprx_bound} using Lemma \ref{lem_GS}. 
First observe that, by coercivity of $\newb $, the Lax--Milgram lemma (see, e.g., \cite[Lemma 2.32]{Mc:00}) and the definitions of $\|\cdot\|_{(H^1_k(\Omega))^*}$ and $\|\cdot\|_{H^1_k(\Omega)}$, 
\beq\label{eq_tsol_bounds}
\big\| \tsol\big\|_{(H^1_k(\Omega))^* \to H^1_k(\Omega)} \leq C,
\qquad 
\big\| \tsol\big\|_{L^2(\Omega) \to H^1_k(\Omega)} \leq Ck^{-1}
\quad\tand\quad
\big\| \tsol\big\|_{L^2(\Omega) \to L^2(\Omega)} \leq Ck^{-2}.
\eeq
Given $\phi\in L^2(\Omega)$, by the definition of $\tsol$ \eqref{eq_tsol}, $v:= \tsol\phi \in H^1_{\GD}(\Omega)$ is the solution of the equation 
\beqs
- k^2  \mu v-\nabla\cdot\big(A\nabla v\big)  + 2k^2 \operator^{2}  v =\phi.
\eeqs
Since $A,\mu$, and $\partial\Omega$ satisfy Assumption \ref{ass_regularity}, 
by elliptic regularity (see, e.g., \cite[Theorem 4.20]{Mc:00}) 
the second bound in \eqref{eq_operator_regularity}, and the second bound in \eqref{eq_tsol_bounds}, 
 for $0\leq m\leq \ell-1$,
\begin{align}\nonumber
|v|_{H^{m+2}(\Part)} \leq C \Big( \N{\nabla \cdot \big( A\nabla v \big)}_{H^m(\Part)}+
\N{v}_{H^1(\Omega)}
\Big)
&\leq C\Big( \N{-2k^2 \operator^{2} v + k^2  \mu v + \phi}_{H^m(\Part)}+
\N{v}_{H^1(\Omega)}
\Big)\\ \nonumber
&\hspace{-0.5cm}\leq C\Big( k^{2+m} \N{v}_{L^2(\Omega)} +k^2 \N{v}_{H^m(\Part)}+ \N{\phi}_{H^m(\Part)}\Big) 
\\
&\hspace{-0.5cm} \leq C\Big( k^m \N{\phi}_{L^2(\Omega)} + k^2 \N{v}_{H^m(\Part)} + \N{\phi}_{H^m(\Part)}\Big).
\label{eq_really_the_end}
\end{align}
After multiplying by $k^{-m-1}$ and using the definitions of the weighted norms \eqref{eq_weighted} and \eqref{eq_weighted2}, we obtain that 
\beq\label{eq_ER2}
k^{-m-1}|v|_{H^{m+2}(\Part)} \leq C \Big( \N{v}_{H^m_k(\Part)} + k^{-2} \N{\phi}_{H^m_k(\Part)}\Big)\quad\tfor
1\leq m\leq \ell-1.
\eeq
By iteratively applying \eqref{eq_ER2}, and then using either  the second and third bounds in \eqref{eq_tsol_bounds} (depending on whether $m$ is odd or even), we obtain that $\N{v}_{H^{n+2}_k(\Part)}\leq k^{-2} \N{\phi}_{H^{n}_k(\Part)}$ for all $1\leq n\leq \ell-1$; 
with $n=\ell-1$ this is
 the result \eqref{eq_tsol_regularity} for $\ell\geq 2$. The bound \eqref{eq_tsol_regularity} for $\ell=1$ follows from taking $m=0$ in \eqref{eq_really_the_end}, multiplying by $k^{-1}$, and then using 
the bounds in \eqref{eq_tsol_bounds}.
\epf

\subsection{Transferring the error bound from the mapped finite-element space to the true finite-element space}\label{sec:mapped2}

With Theorems \ref{thm_gammaq} and \ref{thm_main_event1}, 
we have all the tools needed to apply the abstract framework of
Section \ref{sec:abstract} to obtain an error estimate for
the abstract mapped finite-element solution, i.e., a bound on
$\|u-\tu_h\|_{H^1_k(\Omega)}$. We now provide an additional set of
results  to give an error estimate for the ``concrete''
finite element solution $u_h$. Since we can expect that $K = \tK$
for the elements $K \in \CT_h$ not touching an interface, the
error bound for $\tu_h$ already gives an error bound for $u_h$
on those elements, but we can be a bit more precise. To do so,
we need the following additional assumption that essentially
requires that for the elements $K \in \CT_h$ not touching an
interface, $\Psi^h_K = I$ ; see Figure~\ref{figure_pathology}.

\begin{assumption}[Assumption of the behaviour of $\Psi^h$ inside $Q$]
\label{assumption_straight_fem}
For all $Q \in \Part$, 
\begin{equation*}
Q \cap Q^h = \bigcup_{\substack{K \in \CT_h \\ K \subset \overline{Q^h}}} \tK \cap K.
\end{equation*}
\end{assumption}

\input{figures/figure_pathology}

Under Assumption \ref{assumption_straight_fem}, we can transform
an error estimate on $\tu_h$ into an error estimate on $u_h$.

\begin{lemma}[Error transfer]\label{lem_error_transfer}
Let $v \in H^1_{\GD}(\Omega)$ and $v_h \in V_h$. Then, for all $Q \in \Part$, 
\begin{equation}
\label{eq_transform_error}
\|v-v_h\|_{H^1_k(Q \cap Q^h)}
\leq
C \left (
\|v-\tv_h\|_{H^1_k(Q)}
+
\left (\frac{h}{L}\right )^q\|v\|_{H^1_k(Q)}
\right).
\end{equation}
\end{lemma}

\begin{proof}
If we can prove that 
\begin{equation}\label{eq_lastday5}
\|\tv_h-v_h\|_{H^1_k(Q \cap Q^h)} \leq C \left(\frac{h}{L}\right)^q\|\tv_h\|_{H^1_k(Q \cap Q^h)},
\end{equation}
then the result \eqref{eq_transform_error} follows by the triangle inequality:
\begin{align*}
\|v-v_h\|_{H^1_k(Q \cap Q^h)}
&\leq
\|v-\tv_h\|_{H^1_k(Q \cap Q^h)}
+
\|\tv_h-v_h\|_{H^1_k(Q \cap Q^h)}
\\
&\leq
\|v-\tv_h\|_{H^1_k(Q \cap Q^h)}
+
C \left (\frac{h}{L}\right )^q
\|\tv_h\|_{H^1_k(Q \cap Q^h)}
\\
&\leq
\left (1 + C \left (\frac{h}{L}\right )^q\right )
\|v-\tv_h\|_{H^1_k(Q \cap Q^h)}
+
C \left (\frac{h}{L}\right )^q
\|v\|_{H^1_k(Q \cap Q^h)}
\\
&\leq
C \left (
\|v-\tv_h\|_{H^1_k(Q \cap Q^h)}
+
\left (\frac{h}{L}\right )^q
\|v\|_{H^1_k(Q \cap Q^h)}
\right )
\\
&\leq
C \left (
\|v-\tv_h\|_{H^1_k(Q)}
+
\left (\frac{h}{L}\right )^q
\|v\|_{H^1_k(Q)}
\right ).
\end{align*}
We now prove \eqref{eq_lastday5}; by Assumption \ref{assumption_straight_fem}, 
\begin{equation*}
\|\tv_h-v_h\|_{L^2(Q \cap Q^h)}^2
=
\sum_{K \in \CT_h \,,\,K \subset \overline{Q^h}} \|\tv_h-v_h\|_{L^2(K \cap \tK)}^2.
\end{equation*}
We now focus on a single element $K$. Since $\tv_h = v_h \circ \Phi^h$ and 
 $\Phi^h(x) = \Phi_K^h(x)$ when $x\in \tK$,
\begin{align}
\|\tv_h-v_h\|_{L^2(K \cap \tK)}^2
&=
\int_{K \cap \tK} |v_h(\Phi^h(x))-v_h(x)|^2 dx
=
\int_{K \cap \tK} |v_h(\Phi^h_K(x))-v_h(x)|^2 dx.
\label{eq_lastday4}
\end{align}
We now work on the integrand; specifically, by the mean-value inequality,
\begin{align*}
\big|v_h(\Phi^h_K(x))-v_h(x)\big|
&=
\big|v_h(x+(I-\Phi^h_K)(x))-v_h(x)\big|
\\
&\leq
\big|(I-\Phi^k_K)(x)\big|
\|\nabla v_h\|_{L^\infty(K)}
\leq
\|I-\Phi^k_K\|_{L^\infty(K)}
\|\nabla v_h\|_{L^\infty(K)}.
\end{align*}
Using this in \eqref{eq_lastday4} gives 
\begin{equation*}
\|\tv_h-v_h\|^2_{L^2(K \cap \tK)}
\leq
\|I-\Phi^k_K\|_{L^\infty(K)}^2
|K|\|\nabla v_h\|_{L^\infty(K)}^2
\leq
Ch_K^{-2} \|I-\Phi^k_K\|_{L^\infty(K)}^2\|v_h\|_{L^2(K)}^2
\end{equation*}
due to standard equivalence of norms and scaling argument on polynomial spaces recapped in Lemma \ref{lem_scaling} below. %
Summing over $K$ and using the first bound in \eqref{eq_bound_Phi} with $s=0$, we obtain that
\begin{equation}
\label{tmp_transform_L2}
\|\tv_h-v_h\|_{L^2(Q \cap Q^h)}
\leq
C \left (\frac{h}{L}\right )^q  \|v_h\|_{L^2(Q^h)}
\leq                               
C \left (\frac{h}{L}\right )^q  \|\tv_h\|_{L^2(Q)}.
\end{equation}

We argue similarly for the gradient term; since $\tv_h = v_h \circ \Phi^h$ and 
 $\Phi^h(x) = \Phi_K^h(x)$ when $x\in \tK$, by the chain rule (see, e.g., \cite[Equation 9.8a]{ErGu:21}),
\begin{equation*}
\|\nabla(\tv_h-v_h)\|_{L^2(Q \cap Q^h)}^2
=
\sum_{\substack{K \in \CT_h \\ K \subset \overline{Q^h}}}
\int_{K \cap \tK}
\big|(\nabla \Phi^h_K)^T(x)(\nabla v_h)(\Phi^h_K(x))-(\nabla v_h)(x)\big|^2 dx.
\end{equation*}
Then, with $\MI$ the $d\times d$ identity matrix,
\begin{align*}
&\big|(\nabla \Phi^h_K)^T(x)(\nabla v_h)(\Phi^h_K(x))-(\nabla v_h)(x)\big|\\
&\hspace{2cm}\leq
\big|(\MI-\nabla \Phi^h_K)^T(x)(\nabla v_h)(\Phi^h_K(x))\big|
+
\big|(\nabla v_h)(\Phi_h^K(x))-(\nabla v_h)(x)\big|
\\
&\hspace{2cm}\leq
\|\MI -\nabla \Phi^h_K\|_{L^\infty(\tK)} \|\nabla v_h\|_{L^\infty(K)}
+
\|I-\Phi_h^K\|_{L^\infty(\tK)}
\|\nabla^2 v_h\|_{L^\infty(K)},
\end{align*}
by the mean-value inequality. By
the scaling arguments in Lemma \ref{lem_scaling},
\begin{align*}
\|\nabla(\tv_h-v_h)\|_{L^2(K\cap \tK)}
\leq
C
\left (\|\MI -\nabla \Phi^h_K\|_{L^\infty(\tK)}
+
h_K^{-1}\|I-\Phi_h^K\|_{L^\infty(\tK)}
\right )
\|\nabla v_h\|_{L^2(K)}.
\end{align*}
Summing over $K$ and using the first bound in \eqref{eq_bound_Phi} with $s=0$ and $s=1$,
we obtain that
\begin{equation}
\label{tmp_transform_H1}
\|\nabla(\tv_h-v_h)\|_{L^2(Q\cap Q^h)}
\leq
C \left (\frac{h}{L}\right )^q \|\nabla v_h\|_{L^2(Q^h)}
\leq                               
C \left (\frac{h}{L}\right )^q \ \|\nabla \tv_h\|_{L^2(Q)}.
\end{equation}
Combining \eqref{tmp_transform_L2} and \eqref{tmp_transform_H1}, we obtain \eqref{eq_lastday5} and the proof is complete.
\end{proof}

\subsection{Theorem \ref{thm_abs1} applied to Helmholtz problems solved with curved finite elements}\label{sec_main_result_applied}

\begin{theorem}[Preasymptotic Helmholtz $h$-FEM error bound  with curved finite elements]
\label{thm_main_applied}
Suppose that $A,\mu$, and $\data$ satisfy the assumptions in \S\ref{sec_4.1} (in particular, $A$ satisfies Assumption \ref{ass_strong_ellipticity} and $A$, $\mu$, and $\partial\Omega$ satisfy the regularity requirements in Assumption \ref{ass_regularity} for some $\ell\in \mathbb{Z}^+$).
Suppose that $V_h$ satisfies Assumptions \ref{assumption_curved_fem} and \ref{assumption_straight_fem} for some $p\leq \ell$.
Suppose that $kh\leq C_0$ for some $C_0>0$. 

(i) Given $k_0>0$ there exists $C_1, C_2>0$ such that, for all $k\geq k_0$, if
\beq\label{eq_abs_condition1}
\Cstab(kh)^{2p} + (1+\Cstab) \left(\frac{h}{L}\right)^q \leq C_1
\eeq
then the Galerkin solution $u_h$ exists, is unique, and satisfies
\begin{align}\nonumber
&\Big(\sum_{Q\in \Part}\|u-u_h\|_{H^1_k(Q \cap Q^h)}^2\Big)^{1/2}
\\ \nonumber
&\leq C_{2} \bigg[ 
 \Big( 1 + \Cstab (kh)^p\Big) \N{(I- \widetilde{\pi}_h) u}_{H^1_k(\Omega)}
\\
&\hspace{3cm}
 + (1+ \Cstab) \left(\frac{h}{L}\right)^q \bigg(\N{u}_{H^1_k(\Omega)} + \frac{1}{k^2}
 \Big(\sum_{Q \in \Part} \| \data_Q\|^2_{H^1_k(\Rea^d)}\Big)^{1/2}\bigg)
\bigg],
\label{eq_abs_bound}
\end{align}
where $\widetilde{\pi}_h$ is the  orthogonal projection in the $\|\cdot\|_{H^1_k(\Omega)}$ norm in the mapped finite-element space $\widetilde{V}_h$
(which, in particular, satisfies the interpolation bound \eqref{eq_mapped_interpolant}). 

(ii) Suppose, in addition, that 
${\rm supp}\, \data\subset Q$ for some $Q \in \Part$ and that given $k_0>0$ there exists $C>0$ such that
for all $0\leq j\leq p-1$,
\beq\label{eq_f_oscil}
\|\data\|_{H^j(\Omega)}
\leq 
C k^{j+1}
\sup_{\substack{v \in \cH \\ \|v\|_{H^1_k(\Omega)} = 1}} |\langle \data,v\rangle|
\quad\tfa k\geq k_0.
\eeq
Then, given $k_0>0$, there exists $C_3>0$ such that, for all $k\geq k_0$, if \eqref{eq_abs_condition1} holds, then 
\begin{align}\nonumber
&\Big(\sum_{Q\in \Part}\|u-u_h\|_{H^1_k(Q \cap Q^h)}^2\Big)^{1/2}
\leq C_{3} \bigg[ 
 \Big( 1 +  \Cstab(kh)^p\Big) (kh)^p + \Cstab\left(\frac{h}{L}\right)^q
\bigg] \N{u}_{H^1_k(\Omega)}.
\end{align}
\end{theorem}

\bpf
(i) Under the assumptions of the theorem, Theorems \ref{thm_gammaq} and \ref{thm_main_event1} both hold.
Applying Theorem \ref{thm_abs1} to the variational problem \eqref{eq_lastday1} (i.e., in the mapped space $\widetilde{V}_h$)
and inputting the bounds on $\Cquad, \Cdrhs, \Caprx,$ and $\Caprxop$ from Theorems \ref{thm_gammaq} and \ref{thm_main_event1}, we obtain,
for some $C>0$ depending on $A$ and $\mu$,
\begin{align}\nonumber
&\N{u-\widetilde{u}_h}_{H^1_k(\Omega)}\\ \nonumber
&\leq C \big(1 + \Cstab (kh)^p\big) \N{(I- \widetilde{\pi}_h) u}_{H^1_k(\Omega)} \\
&\quad+ C (1+ \Cstab)
\bigg(\left (\frac{h}{L}\right )^q \N{u}_{H^1_k(\Omega)}
+
\left (\frac{h}{L}\right )^q
\frac{1}{k}
\bigg (
\sum_{Q \in \Part}
\|\data_Q\|_{L^2(\mathbb R^d)}^2
+
h^2\|\nabla \data_Q\|_{L^2(\mathbb R^d)}^2
\bigg )^{1/2}
\bigg)\label{eq_abs_bound2}
\end{align}
(compare to \eqref{eq_simple2}). The bound \eqref{eq_transform_error} 
allows us to replace 
$\N{u-\widetilde{u}_h}_{H^1_k(\Omega)}$
on the left-hand side of \eqref{eq_abs_bound2} with $\sum_{Q\in \Part} \N{u-u_h}_{H^1(Q\cup Q^h)}$. The result \eqref{eq_abs_bound} then follows by using $h\leq C_0/k$ in the terms involving $\data_Q$. 

(ii)
The result 
follows from Part (i) if we can show that
\begin{itemize}
\item[(a)] given $k_0>0$ there exists $C>0$ such that, for all $k\geq k_0$, if $kh\leq C_0$, 
\beq\label{eq_GWR0}
\N{(I- \widetilde{\pi}_h) u}_{H^1_k(\Omega)}\leq C (kh)^p\N{u}_{H^1_k(\Omega)},
\eeq
\item[(b)] if $g$ satisfies \eqref{eq_f_oscil}, then 
\beq\label{eq_GWR2}
\frac{1}{k^2}\Big(\sum_{Q \in \Part} \| \data_Q\|_{H^1_k(\Rea^d)}\Big)^{1/2}\leq C\N{u}_{H^1_k(\Omega)}.
\eeq
\end{itemize}
For (a), by \eqref{eq_mapped_interpolant},
\beq\label{eq_GWR1}
\N{(I- \widetilde{\pi}_h) u}_{H^1_k(\Omega)}\leq C (1+kh) h^p \|u\|_{H^{p+1}(\Part)}.
\eeq
If $g$ satisfies \eqref{eq_f_oscil}, then
by elliptic regularity (using the regularity assumptions on $A$, $\mu$, and $\partial\Omega$ in Assumption \ref{ass_regularity}),
\beq\label{eq_u_oscil}
\N{u}_{H^{p+1}(\Part)}\leq C k^{p} \N{u}_{H^1_k(\Omega)} \quad\tfa k\geq k_0;
\eeq
(i.e., the data being $k$-oscillatory implies that the solution is $k$-oscillatory); the proof is very similar to the uses of elliptic regularity in the proof of Lemma \ref{lem_GS}; for the details see \cite[end of the proof of Theorem 1.5]{GaSp:23}. Combining \eqref{eq_GWR1} and \eqref{eq_u_oscil}, we obtain that
\beqs
\N{(I- \widetilde{\pi}_h) u}_{H^1_k(\Omega)}\leq C (1+kh) (kh)^p\N{u}_{H^1_k(\Omega)},
\eeqs
and \eqref{eq_GWR0} follows since $kh\leq C_0$.

For (b), by 
\eqref{eq_f_oscil}  
applied with $j=1$ 
and the fact that ${\rm supp }\data\subset Q$ for some $Q\in \Part$,
\beq\label{eq_GWR3}
\frac{1}{k^2}\Big(\sum_{Q \in \Part} \| \data_Q\|_{H^1_k(\Rea^d)}\Big)^{1/2}
= k^{-2} \|\data\|_{H^1_k(\Omega)}
\leq C\N{\data}_{(H^1_k(\Omega))^*}.
\eeq
By continuity of $b$ \eqref{eq_definition_Ccont} and \eqref{eq:RHS},
\beq\label{eq_GWR4}
\N{\data}_{(H^1_k(\Omega))^*}
 \leq \Ccont \N{u}_{H^1_k(\Omega)},
\eeq
and \eqref{eq_GWR2} follows from combining \eqref{eq_GWR3} and \eqref{eq_GWR4}.
\epf

\begin{remark}[Obtaining \InformalTheorem \ref{thm:informal1} from Part (ii) of Theorem \ref{thm_main_applied}]
For the radial PML originally introduced in \cite{collino_monk_1998a}, the coefficient $A$ satisfies 
Assumption \ref{ass_strong_ellipticity} by, e.g., 
\cite[Lemma 2.3]{GLSW1}. 
\InformalTheorem \ref{thm:informal1} therefore follows from Part (ii) of Theorem \ref{thm_main_applied} if the data satisfies \eqref{eq_f_oscil} and is such that ${\rm supp}\, \data \subset Q$ for some $Q\in \Part$. 

As described in Definitions \ref{def_PML_approx_ss} and  \ref{def_PML_approx_pen}, 
the solution of the plane-wave scattering problem can be approximated using a PML by solving for the unknown 
$u:=\chi u_{\rm inc} +u_{\rm scat}$, with $\chi \in C^\infty_{\rm comp}(\Rea^d)$ equal one on the scatterer,
and corresponding data $\data=\datai$ with $\datai$ defined by \eqref{eq_new_rhs}. 
By construction, the PML must be away from the support of $\chi$; thus
$\nabla \chi$, and hence also $\data$, are supported where $A=I$ and $\mu=1$, and thus on a single $Q\in \Part$.
\end{remark}

\begin{appendix}

\section{Appendix}
\subsection{Scaling arguments and inverse estimates}
\label{section_scaling_argument}

In this section, we consider a mesh $\CT_h$
satisfying the requirements of Assumption \ref{assumption_curved_fem}.
Recall that, in this case, the elements $K \in \CT_h$ are obtained by
mapping a single reference simplex $\hK$ through bilipschitz maps $\MAP_K: \hK \to K$.
To simplify the notation, we let $\IMAP_K \eq \MAP_K^{-1}$. 
By \eqref{eq_mapK}, 
\begin{subequations}
\label{eq_determinant}
\begin{equation}
\|\det \nabla \MAP_K\|_{L^\infty(K)}
\leq
C\|\nabla \MAP_K\|_{L^\infty(K)}^d
\leq
C h_K^d,
\end{equation}
and
\begin{equation}
\|\det \nabla \IMAP_K\|_{L^\infty(\hK)}
\leq
C\|\nabla \IMAP_K\|_{L^\infty(\hK)}^d
\leq
C h_K^{-d}.
\end{equation}
\end{subequations}

\begin{lemma}\label{lem_scaling}
For all $v_h \in V_h$ and $K \in \CT_h$,
\begin{equation*}
|K| \|\nabla v_h\|_{L^\infty(K)}^2 \leq C h_K^{-2} \|v_h\|_{L^2(K)}^2,
\quad
|K| \|\nabla^2 v_h\|_{L^\infty(K)}^2 \leq C h_K^{-2} \|\nabla v_h\|_{L^2(K)}^2.
\end{equation*}
\end{lemma}

\begin{proof}
By \eqref{eq_determinant}, 
\begin{equation}\label{eq_measureK}
|K|
=
\int_K 1 dx
=
\int_{\hK} (\det \nabla \MAP_K)(\hx) d\hx
\leq
|\hK| \|\det \nabla \MAP_K\|_{L^\infty(\hK)}
\leq
C h_K^d.
\end{equation}
Let $\hv \in \CP_p(\hK)$ and set $v \eq \hv \circ \MAP_K^{-1}
= \hv \circ \IMAP_K$.
By the chain rule, 
\begin{equation*}
\partial_j v
=
\partial_j \IMAP_K^\ell (\partial_\ell \hv) \circ \IMAP_K
\end{equation*}
and a second application gives
\begin{equation*}
\partial_{j\ell}^2 v
=
\partial_{j\ell}^2 \IMAP_K^\ell (\partial_\ell \hv) \circ \IMAP_K
+
\partial_j \IMAP_K^\ell \partial_\ell \IMAP_K^m (\partial_{\ell m}^2 \hv) \circ \IMAP_K.
\end{equation*}
It then follows from \eqref{eq_mapK}
that
\begin{equation*}
\|\nabla v\|_{L^\infty(K)}
\leq
C h_K^{-1} \|\nabla \hv\|_{L^\infty(\hK)}
\quad\tand\quad
\|\nabla^2 v\|_{L^\infty(K)}
\leq
C h_K^{-2} \left (
\|\nabla \hv\|_{L^\infty(\hK)}
+
\|\nabla^2 \hv\|_{L^\infty(\hK)}
\right ).
\end{equation*}
Since $\CP_p(\hK)$ is a finite dimensional space,
equivalence of semi-norms gives that
\begin{equation*}
\|\nabla \hv\|_{L^\infty(\hK)}
\leq
C \|\hv\|_{L^2(\hK)}
\quad\tand\quad
\|\nabla \hv\|_{L^\infty(\hK)} + \|\nabla^2 \hv\|_{L^\infty(\hK)}
\leq
C \|\nabla \hv\|_{L^2(\hK)},
\end{equation*}
leading to
\begin{equation}\label{eq_Friday1}
\|\nabla v\|_{L^\infty(K)}
\leq
C h_K^{-1} \|\hv\|_{L^2(\hK)}
\quad\tand\quad
\|\nabla^2 v\|_{L^\infty(K)}
\leq
C h_K^{-2} \|\nabla \hv\|_{L^2(\hK)}.
\end{equation}

By the chain rule in the other direction,
\begin{equation*}
|\nabla \hv|^2
\leq
C h_K^2 |\nabla v \circ \MAP_K|^2,
\end{equation*}
and, by the change-of-variable formula,
\begin{equation}\label{eq_Friday2}
\|\hv\|_{L^2(\hK)}^2 \leq C h_K^{-d} \|v\|_{L^2(K)}^2
\quad\tand\quad
\|\nabla \hv\|_{L^2(\hK)}^2 \leq C h_K^{-d} h_K^2 \|\nabla v\|_{L^2(K)}^2,
\end{equation}
where we again used \eqref{eq_determinant}; combining \eqref{eq_Friday1}, \eqref{eq_Friday2}, and \eqref{eq_measureK} concludes the proof.
\end{proof}


\subsection{Composition with a map close to the identity}
\label{section_change_variable}

The following result is used to estimate the geometric error on the right-hand side (in Theorem \ref{thm_gammaq}).

\begin{lemma}\label{lemma_change_variable}
Let $U \subset \mathbb R^d$ 
be a measurable set and 
$\phi: U \to \phi(U)$  a bilipschitz mapping 
such that $\|I-\phi\|_{L^\infty(U)} \leq \varepsilon$ and
$\|\MI -\nabla \phi\|_{L^\infty(U)} \leq \varepsilon' < 1$.
Let $U_\varepsilon := \{ x \in \mathbb R^d  :  \operatorname{dist}(x,U) < \varepsilon$ \}
and $f \in H^1(U_\varepsilon)$. Then
\begin{equation*}
\|f-f\circ\phi\|_{L^2(U)}
\leq
\frac{\varepsilon}{1-\varepsilon'}\|\nabla f\|_{L^2(U_\varepsilon)}.
\end{equation*}
\end{lemma}

\begin{proof}
Let  $\phi_t(x) = x + t(\phi(x)-x)$ and $g_x(t) = f(\phi_t(x))$ $t \in [0,1]$.
Then
\begin{equation*}
f(\phi(x))-f(x)
=
\int_0^1 g_x'(t) dt
=
\int_0^1 (\phi(x)-x) \cdot (\nabla f) (\phi_t(x)) dt,
\end{equation*}
and
\begin{align*}
\int_U |f(\phi(x))-f(x)|^2 dx
&=
\int_U
\left |\int_0^1 (\phi(x)-x) \cdot (\nabla f) (\phi_t(x)) dt\right |^2
dx\\
&\leq
\int_0^1
\int_U \big|(\phi(x)-x) \cdot (\nabla f) (\phi_t(x))\big|^2 dx
dt,
\end{align*}
where we have used the Cauchy-Schwarz inequality and Fubini's theorem.

Since $\nabla \phi_t = \MI  + t(\nabla \phi-\MI )$, we have
$|\nabla \phi_t(x)| \geq 1 - t\varepsilon' \geq 1-\varepsilon'$.
It follows that $\phi_t^{-1}$ exists for all $t \in [0,1]$ with
\begin{equation}
\label{tmp_inverse_gradient}
|\nabla (\phi_t^{-1})(y)| \leq \frac{1}{1-\varepsilon'}
\qquad
\tfa y \in \phi_t(U).
\end{equation}
Hence, for a fixed $t \in [0,1]$,
by \eqref{tmp_inverse_gradient} and the fact that
$\phi_t(U) \subset U_\varepsilon$ for all $t \in [0,1]$,
\begin{align*}
\int_U \big|(\phi(x)-x) \cdot (\nabla f) (\phi_t(x))\big|^2 dx
&\leq
\|\phi-I\|_{L^\infty(U)}^2
\int_U |(\nabla f)(\phi_t(x))|^2 dx
\\
&=
\|\phi-I\|_{L^\infty(U)}^2
\int_{\phi_t(U)} \big|\det \nabla (\phi_t^{-1})(y)|^2|(\nabla f)(y)\big|^2 dy
\\
&\leq
\left (\frac{\varepsilon}{1-\varepsilon'}\right )^2 \int_{U_\varepsilon} |(\nabla f)(y)|^2 dy,
\end{align*}
and the result follows.
\end{proof}
\end{appendix}

\section*{Acknowledgements}
For useful discussions and comments, the authors thank Jeffrey Galkowski (University College London) and the anonymous referees.

\bibliographystyle{IMANUM-BIB}
\bibliography{biblio}

\end{document}